\pdfoutput=1

\documentclass[a4paper,12pt]{article}

\usepackage[french]{babel}
\usepackage[utf8]{inputenc}
\usepackage[T1]{fontenc}
\usepackage{lmodern}
\usepackage{amsmath,amssymb,amsthm}
\usepackage[linktocpage=false]{hyperref}

\DeclareFontFamily{OMX}{lmex}{}
\DeclareFontShape{OMX}{lmex}{m}{n}{<->lmex10}{}

\theoremstyle{plain}
\newtheorem{theo}{Théorème}[section]
\newtheorem{prop}[theo]{Proposition}
\newtheorem{conj}[theo]{Conjecture}
\newtheorem{coro}[theo]{Corollaire}
\newtheorem{lemm}[theo]{Lemme}

\theoremstyle{definition}
\newtheorem{defi}[theo]{Définition}
\theoremstyle{remark}
\newtheorem{rema}[theo]{Remarque}
\newtheorem{exem}[theo]{Exemple}

\renewcommand{\leq}{\leqslant}
\renewcommand{\geq}{\geqslant}

\DeclareMathOperator{\card}{card}
\DeclareMathOperator{\val}{val}
\DeclareMathOperator{\Hom}{Hom}
\DeclareMathOperator{\Ext}{Ext}
\DeclareMathOperator{\Ind}{Ind}
\DeclareMathOperator{\cind}{c-ind}
\DeclareMathOperator{\Ord}{Ord}
\DeclareMathOperator{\Lie}{Lie}
\DeclareMathOperator{\Gal}{Gal}
\DeclareMathOperator{\Fil}{Fil}
\DeclareMathOperator{\Gr}{Gr}
\DeclareMathOperator{\Nrm}{Nrm}
\DeclareMathOperator{\Res}{Res}
\DeclareMathOperator{\soc}{soc}
\DeclareMathOperator{\rad}{rad}
\DeclareMathOperator{\detfr}{d\acute{e}t}
\newcommand{\llbrack}{[\![}
\newcommand{\rrbrack}{]\!]}
\newcommand{\dfn}{\overset{\text{\upshape déf}}{=}}
\newcommand{\iso}{\overset{\sim}{\longrightarrow}}
\newcommand{\h}{\overset{\mathrm{H}}{\cdot}}
\newcommand{\sms}{\mathrm{ss}}
\newcommand{\ord}{\mathrm{ord}}
\newcommand{\fin}{\mathrm{fin}}
\newcommand{\N}{\mathbb{N}}
\newcommand{\Z}{\mathbb{Z}}
\newcommand{\Q}{\mathbb{Q}}
\newcommand{\Fp}{\mathbb{F}_p}
\newcommand{\Fpbar}{\overline{\Fp}}
\newcommand{\Zp}{\Z_p}
\newcommand{\Qp}{\Q_p}
\newcommand{\Qpbar}{\overline{\Qp}}
\newcommand{\Oe}{\mathcal{O}_{\!E}}
\newcommand{\pe}{\varpi_{\!E}}
\newcommand{\ke}{k_E}
\newcommand{\A}[1]{\Oe/\pe^{#1}\Oe}
\newcommand{\GL}{\mathrm{GL}}
\newcommand{\GSp}{\mathrm{GSp}}
\newcommand{\Gd}{\mathrm{G}_2}
\newcommand{\oma}{\omega^{-1} \circ \alpha}
\newcommand{\epsth}{\cdot (\varepsilon^{-1} \circ \theta)}
\newcommand{\omth}{\cdot (\omega^{-1} \circ \theta)}
\newcommand{\Pt}{\widetilde{P}}
\newcommand{\Lt}{\widetilde{L}}
\newcommand{\Nt}{\widetilde{N}}
\newcommand{\nt}{\widetilde{n}}
\newcommand{\Ntz}{\Nt_0}
\newcommand{\Ntlz}[1][l]{\Nt_{#1,0}}
\newcommand{\NPz}{N_{P,0}}
\newcommand{\NPw}{N_{P,w}}
\newcommand{\NPwz}{N_{P,w,0}}
\newcommand{\Ns}[1]{N_{s_{#1} \dots s_1}}
\newcommand{\NPs}[1]{N_{P,s_{#1} \dots s_1 w_L}}
\newcommand{\NPsz}[1]{N_{P,s_{#1} \dots s_1 w_L,0}}
\newcommand{\NLw}{N_{L,w_L}}
\newcommand{\Ut}{\widetilde{U}}
\newcommand{\Vwz}{V_{w,0}}
\newcommand{\Hc}[1][\bullet]{\mathrm{H}^{#1}}
\newcommand{\OrdF}[1][P]{\Ord_{#1(F)}}

\newcommand{\HOrdF}[1][\bullet]{\Hc[#1]\!\Ord_{P(F)}}
\newcommand{\HOrdQp}[1][P]{\Hc[1]\!\Ord_{#1(\Qp)}}
\newcommand{\IndF}[1][B^-]{\Ind_{#1(F)}^{G(F)}}
\newcommand{\IndFL}{\Ind_{B_L^-(F)}^{L(F)}}
\newcommand{\IndQp}[1][B^-]{\Ind_{#1(\Qp)}^{G(\Qp)}}
\newcommand{\IndQpL}{\Ind_{B_L^-(\Qp)}^{L(\Qp)}}
\newcommand{\Inda}{\Ind_{B_\alpha^-(\Qp)}^{G_\alpha(\Qp)}}
\newcommand{\Indd}{\Ind_{\left(\begin{smallmatrix} * & 0 \\ * & * \end{smallmatrix}\right)}^{\GL_2(\Qp)}}
\newcommand{\cindF}{\cind_{B^-(F)}}
\newcommand{\cindFL}{\cind_{B_L^-(F)}}
\newcommand{\cindQp}{\cind_{B^-(\Qp)}}
\newcommand{\C}{\mathcal{C}}
\newcommand{\Clis}{\C^\infty}
\newcommand{\Clisc}{\Clis_\mathrm{c}}
\newcommand{\Ct}{\widetilde{\C}}
\newcommand{\E}{\mathcal{E}}
\newcommand{\Et}{\widetilde{\E}}
\newcommand{\Wt}{\widetilde{W}}
\newcommand{\wt}{\widetilde{w}}
\newcommand{\dw}{\dot{w}}
\newcommand{\dwt}{\dot{\wt}}
\newcommand{\ds}{\dot{s}}
\newcommand{\alphat}{\widetilde{\alpha}}

\hypersetup{pdfinfo={
	Title={Sur une conjecture de Breuil-Herzig},
	Author={Julien Hauseux},
	Subject={22E50},
	Keywords={programme de Langlands p-adique, groupes réductifs p-adiques, séries principales, extensions, parties ordinaires dérivées, filtration de Bruhat}
}}

\title{Sur une conjecture de Breuil-Herzig}
\author{Julien Hauseux}
\date{}

\begin{document}

\maketitle

\begin{abstract}
Soit $G$ un groupe réductif $p$-adique de centre connexe et de groupe dérivé simplement connexe.
Nous montrons que certaines \og chaînes \fg{} de séries principales de $G$ n'existent pas et nous établissons plusieurs propriétés de la construction $\Pi(\rho)^\mathrm{ord}$ de Breuil-Herzig.
En particulier, nous obtenons une caractérisation naturelle de cette dernière et nous démontrons une conjecture de Breuil-Herzig.
Pour cela, nous calculons le $\delta$-foncteur $\mathrm{H^\bullet Ord}_P$ des parties ordinaires dérivées d'Emerton relatif à un sous-groupe parabolique $P$ de $G$ sur une série principale.
Nous énonçons une nouvelle conjecture sur les extensions entre représentations lisses modulo $p$ de $G$ obtenues par induction parabolique à partir de représentations supersingulières de sous-groupes de Levi de $G$ et nous la démontrons pour les extensions par une série principale.
\end{abstract}

\tableofcontents

\section{Introduction}

\subsection*{Contexte}

Soient $F$ une extension finie de $\Qp$ et $G$ un groupe algébrique connexe réductif déployé sur $F$.
On fait les hypothèses suivantes sur $G$ : son centre est connexe et son groupe dérivé est simplement connexe (par exemple $G=\GL_n$ ou $G=\GSp_{2n}$).
On note $\widehat{G}$ le groupe dual de $G$.

Une correspondance de Langlands $p$-adique (resp. modulo $p$) devrait associer à toute représentation continue $\rho : \Gal(\Qpbar/F) \to \widehat{G}(E)$ avec $E$ une extension finie de $\Qp$ (resp. $\overline{\rho} : \Gal(\Qpbar/F) \to \widehat{G}(\ke)$ avec $\ke$ le corps résiduel de $E$), une ou plusieurs représentations continues unitaires $\Pi(\rho)$ de $G(F)$ sur des $E$-espaces de Banach (resp. une ou plusieurs représentations lisses $\Pi(\overline{\rho})$ de $G(F)$ sur des $\ke$-espaces vectoriels).

Lorsque $F=\Qp$ et $\rho$ (resp. $\overline{\rho}$) est ordinaire (c'est-à-dire à valeurs dans un sous-groupe de Borel) et suffisamment générique, Breuil et Herzig (\cite{BH}) construisent une représentation continue unitaire $\Pi(\rho)^\mathrm{ord}$ de $G(\Qp)$ sur un $E$-espace de Banach (resp. une représentation lisse $\Pi(\overline{\rho})^\mathrm{ord}$ de $G(\Qp)$ sur un $\ke$-espace vectoriel) qui devrait être la plus grande sous-représentation fermée de $\Pi(\rho)$ (resp. $\Pi(\overline{\rho})$) dont les constituants irréductibles sont des sous-quotients de séries principales.
Ils conjecturent l'unicité des facteurs directs indécomposables de cette représentation étant donnés les gradués de leurs filtrations par le socle (\cite[Conjecture 3.5.1]{BH}).

\subsection*{Principaux résultats}

Soient $B \subset G$ un sous-groupe de Borel et $T \subset B$ un tore maximal déployé.
On note $B^- \subset G$ le sous-groupe de Borel opposé à $B$ par rapport à $T$, $W$ le groupe de Weyl de $(G,T)$, $\Phi^+$ les racines positives de $(G,B,T)$ et $\Delta$ les racines simples de $\Phi^+$.
Pour tout $\alpha \in \Phi^+$, on note $\alpha^\vee$ la coracine correspondante et $s_\alpha \in W$ la réflexion correspondante.
On note $\theta$ la somme des poids fondamentaux (bien définie à un caractère algébrique de $G$ près, voir \cite[Proposition 2.1.1]{BH}), $\varepsilon : F^\times \to \Zp^\times$ le caractère cyclotomique $p$-adique et $\Oe$ l'anneau des entiers de $E$.

\medskip

Lorsque $F=\Qp$, nous montrons que certaines \og chaînes \fg{} de trois séries principales (en particulier toutes celles sans multiplicité) n'existent pas dans la catégorie des représentations continues unitaires admissibles de $G(\Qp)$ sur $E$ (Théorème \ref{theo:chaine}).
Le résultat analogue modulo $p$ (c'est-à-dire dans la catégorie des représentations lisses admissibles de $G(\Qp)$ sur $\ke$) se démontre de façon analogue.

\begin{theo} \label{theo:1}
Soient $\chi, \chi', \chi'' : T(\Qp) \to \Oe^\times \subset E^\times$ des caractères continus unitaires.
On suppose $\chi \neq \chi''$ et si $\chi'=\chi$ ou $\chi'=\chi''$, alors $\chi'$ faiblement générique (Définition \ref{defi:gen}).
Alors il n'existe pas de représentation continue unitaire admissible de $G(\Qp)$ sur $E$ ayant $\IndQp \chi \epsth$ pour socle, $\IndQp \chi'' \epsth$ pour cosocle et $\IndQp \chi' \epsth$ pour unique constituant irréductible intermédiaire.
\end{theo}

Le théorème \ref{theo:1}, conjointement avec \cite[Théorème 1.1]{JH}, nous permet d'établir certaines propriétés de la construction de Breuil-Herzig.
Les résultats analogues modulo $p$ se démontrent de façon analogue.
Si l'on se donne une représentation galoisienne ordinaire générique $\rho$, alors $\Pi(\rho)^\ord$ est définie à partir d'un caractère continu unitaire $\chi_\rho : T(\Qp) \to \Oe^\times \subset E^\times$ et d'un sous-ensemble fermé $\Psi_\rho \subset \Phi^+$ qui lui sont associés (voir \cite[§~3.3]{BH} où $\Psi_\rho$ est noté $C_\rho$).

Plus généralement, on se donne un caractère continu unitaire $\chi : T(\Qp) \to \Oe^\times \subset E^\times$ et un sous-ensemble fermé $\Psi \subset \Phi^+$.
On note $W_\Psi \subset W$ le sous-ensemble constitué des éléments $w \in W$ vérifiant $w(\Psi) \subset \Phi^+$.
On suppose que $\chi$ est générique (c'est-à-dire $\chi \circ \alpha^\vee \neq 1$ pour tout $\alpha \in \Phi^+$) et que pour tout $w_\Psi \in W_\Psi$ et pour tout $I \subset \Delta \cap w_\Psi(\Psi)$ constitué de racines deux à deux orthogonales, la représentation
\begin{equation} \label{const}
\IndQp \big( \prod_{\alpha \in I} s_\alpha \big) w_\Psi(\chi) \epsth
\end{equation}
est topologiquement irréductible.
D'après \cite[Conjecture 3.1.2]{BH}, ces représentations devraient être topologiquement irréductibles lorsque $\chi \circ \alpha^\vee \neq \varepsilon^{\pm1}$ pour tout $\alpha \in \Phi^+$.
La conjecture analogue modulo $p$ est vraie d'après \cite[Théorème 4]{Oll} lorsque $G=\GL_n$ et \cite[Theorem 1.3]{Abe} dans le cas général déployé.
En particulier si la réduction $\overline{\chi} : T(\Qp) \to \ke^\times$ vérifie $\overline{\chi} \circ \alpha^\vee \neq \omega^{\pm 1}$ pour tout $\alpha \in \Phi^+$, alors ces représentations sont topologiquement irréductibles.
Par la construction de Breuil-Herzig, on obtient une représentation continue unitaire admissible de $G(\Qp)$ sur $E$
\begin{equation*}
\Pi(\chi)_\Psi = \bigoplus_{w_\Psi \in W_\Psi} \Pi(\chi)_{\Psi,w_\Psi}
\end{equation*}
où pour tout $w_\Psi \in W_\Psi$, $\Pi(\chi)_{\Psi,w_\Psi}$ est indécomposable, de longueur finie, sans multiplicité et ses constituants irréductibles sont exactement les représentations \eqref{const} avec $I \subset \Delta \cap w_\Psi(\Psi)$ parmi les sous-ensembles de racines deux à deux orthogonales.
On démontre une caractérisation de ces facteurs directs indécomposables (Théorème \ref{theo:ord}).

\begin{theo} \label{theo:2}
Soient $\chi : T(\Qp) \to \Oe^\times \subset E^\times$ un caractère continu unitaire générique, $\Psi \subset \Phi^+$ un sous-ensemble fermé et $w_\Psi \in W_\Psi$.
On suppose que pour tout $I \subset \Delta \cap w_\Psi(\Psi)$ constitué de racines deux à deux orthogonales, la représentation \eqref{const} est topologiquement irréductible.
Alors $\Pi(\chi)_{\Psi,w_\Psi}$ est la plus grande représentation continue unitaire admissible de $G(\Qp)$ sur $E$ dont le socle est $\IndQp w_\Psi(\chi) \epsth$ et dont les autres sous-quotients irréductibles sont des séries principales distinctes de $\IndQp w'_\Psi(\chi) \epsth$ pour tout $w'_\Psi \in W_\Psi$.
\end{theo}

Le théorème \ref{theo:2} avec $\Psi=\Phi^+$ donne une classification de toutes les représentations continues unitaires admissibles de $G(\Qp)$ sur $E$ dont le socle est $\IndQp \chi \epsth$ et dont les autres sous-quotients irréductibles sont des séries principales distinctes du socle : ce sont exactement les sous-représentations fermées de $\Pi(\chi)_{\Phi^+,1}$.
De plus, on déduit du théorème \ref{theo:2} que pour tout $w_\Psi \in W_\Psi$, la représentation $\Pi(\chi)_{\Psi,w_\Psi}$ est unique étant donnée son socle et ses constituants irréductibles (avec multiplicité un).
En particulier, on prouve la conjecture de Breuil-Herzig (Corollaire \ref{coro:conjBH}).

\begin{coro}
\cite[Conjecture 3.5.1]{BH} est vraie.
\end{coro}

Enfin, on déduit du théorème \ref{theo:2} une caractérisation naturelle de la construction de Breuil-Herzig (Corollaire \ref{coro:ord}).

\begin{coro}
Soient $\chi : T(\Qp) \to \Oe^\times \subset E^\times$ un caractère continu unitaire fortement générique (Définition \ref{defi:gen}) et $\Psi \subset \Phi^+$ un sous-ensemble fermé.
On suppose que pour tout $w_\Psi \in W_\Psi$ et pour tout $I \subset \Delta \cap w_\Psi(\Psi)$ constitué de racines deux à deux orthogonales, la représentation \eqref{const} est topologiquement irréductible.
Alors $\Pi(\chi)_\Psi$ est la plus grande représentation continue unitaire admissible de $G(\Qp)$ sur $E$ satisfaisant
\begin{enumerate}
\item $\soc \Pi(\chi)_\Psi \cong \bigoplus_{w_\Psi \in W_\Psi} \IndQp w_\Psi(\chi) \epsth$ ;
\item les sous-quotients irréductibles de $\Pi(\chi)_\Psi$ sont des séries principales ;
\item les facteurs irréductibles de $\soc \Pi(\chi)_\Psi$ apparaissent avec multiplicité un dans $\Pi(\chi)_\Psi$ (donc seulement dans le socle).
\end{enumerate}
\end{coro}

Nous terminons cet article en énonçant une nouvelle conjecture (suggérée par Breuil lorsque $G=\GL_n$) sur les extensions entre les représentations lisses irréductibles de $G(F)$ sur $\ke$ obtenues par induction parabolique à partir de représentations supersingulières de sous-groupes de Levi (Conjecture \ref{conj:ext}).
Nous démontrons cette conjecture pour les extensions par une série principale suffisamment générique lorsque $F=\Qp$ (Proposition \ref{prop:conjQp}) et quelconque lorsque $F \neq \Qp$ (Proposition \ref{prop:conjF}).

\subsection*{Méthodes utilisées}

La preuve du théorème \ref{theo:1} est un calcul d'extensions.
On procède par réduction modulo $\pe^k$ et dévissage (avec $\pe \in \Oe$ une uniformisante) et en caractéristique positive, on suit une stratégie d'Emerton (\cite{Em2}).
Lorsque $\chi'=\chi$ ou $\chi'=\chi''$, on utilise les calculs de parties ordinaires dérivées de \cite{JH}.
Lorsque les caractères sont deux à deux distincts, il faut calculer le foncteur $\HOrdQp[P_\alpha]$ sur une série principale avec $P_\alpha \subset G$ le sous-groupe parabolique standard correspondant à une racine simple $\alpha \in \Delta$.

\medskip

Plus généralement, sans hypothèse sur $F$ ou $G$, nous calculons partiellement le $\delta$-fonteur $\HOrdF$ relatif à un sous-groupe parabolique standard $P \subset G$ sur une série principale lisses à coefficients dans une $\Oe$-algèbre locale artinienne $A$ de corps résiduel $\ke$.
On note $L \subset P$ le sous-groupe de Levi standard, $B_L \subset L$ (resp. $B_L^- \subset L$) le sous-groupe de Borel $B \cap L$ (resp. $B^- \cap L$), $W_L \subset W$ le groupe de Weyl de $(L,T)$ et $\omega : F^\times \to A^\times$ l'image de $\varepsilon$ dans $A^\times$.
On conjecture le résultat suivant (Conjecture \ref{conj:HnOrd}).

\begin{conj}
Soit $U$ une représentation lisse localement admissible de $T(F)$ sur $A$.
Pour tout $n \in \N$, on a un isomorphisme naturel $L(F)$-équivariant
\begin{equation*}
\HOrdF[n] \left( \IndF U \right) \cong \bigoplus_{[F:\Qp] \cdot \ell(\wt_P)=n} \Ind_{B_L^-(F)}^{L(F)} \left( U^{\wt_P} \otimes (\oma_{\wt_P}) \right)
\end{equation*}
avec $\wt_P$ parmi les représentants de longueur minimale des classes à gauche $W/W_L$, $\ell(\wt_P)$ la longueur de $\wt_P$, $U^{\wt_P}$ la représentation $U$ conjuguée par $\wt_P$ et $\alpha_{\wt_P}$ la somme des racines positives $\alpha \in \Phi^+$ telles que $\wt_P(\alpha) \not \in \Phi^+$.
\end{conj}

Nous démontrons cette conjecture d'une part lorsque les termes de la somme directe sont irréductibles (Proposition \ref{prop:HnOrdirr}) et d'autre part lorsque $[F:\Qp] \nmid n$ (Corollaire \ref{coro:HnOrdF}) auquel cas la somme directe est nulle.
Dans le cas général, nous construisons un morphisme naturel $L(F)$-équivariant entre les représentations de cette conjecture et nous montrons que ces dernières sont naturellement munies de filtrations dont les gradués sont naturellement isomorphes en tant que représentations de sous-groupes fermés de $B_L(F)$ (Théorème \ref{theo:HnOrd}).
Lorsque $F=\Qp$, $P=P_\alpha$ et $n=1$, ce dernier résultat nous suffit pour démontrer le théorème \ref{theo:1}.

\medskip

Expliquons l'organisation de ces calculs de parties ordinaires dérivées (Section \ref{sec:POD}).
On note $N_P$ le radical unipotent de $P$.
Les foncteurs $\HOrdF$ sont construits à partir des $A$-modules de cohomologie d'un sous-groupe ouvert compact $\NPz \subset N_P(F)$ et de l'action de Hecke d'un sous-monoïde ouvert $L^+ \subset L(F)$ sur ces derniers.

La sous-section \ref{ssec:hecke} est indépendante et généralise \cite[§~3.1 et §~3.2]{JH}.
En particulier, les groupes algébriques considérés ne sont pas nécessairement des sous-groupes fermés de $G$ et les dévissages considérés ne sont pas nécessairement des produits semi-directs.
Dans la sous-section \ref{ssec:fil}, on définit à partir de la décomposition de Bruhat des filtrations de représentations induites et on calcule leurs gradués.
Dans la sous-section \ref{ssec:calc}, on calcule partiellement la cohomologie de $\NPz$ à valeur dans ces gradués ainsi que l'action de Hecke de $L^+$.
Pour cela, on utilise des dévissages successifs de $\NPz$ et un nouveau résultat clé (Lemme \ref{lemm:H0}).
Dans la sous-section \ref{ssec:HnOrd}, on montre que les filtrations de Bruhat induisent des filtrations des parties ordinaires dérivées et on déduit des calculs précédents les résultats susmentionnés.

\subsection*{Notations et conventions}

Soit $F$ une extension finie de $\Qp$.
On note $\lvert~\rvert_p : \Qp^\times \to p^\Z$ la valeur absolue $p$-adique, $\varepsilon : F^\times \to \Zp^\times$ le caractère cyclotomique $p$-adique (défini par $\varepsilon(x) = \Nrm_{F/\Qp}(x) \lvert \Nrm_{F/\Qp}(x) \rvert_p$ pour tout $x \in F^\times$) et $\omega : F^\times \to \Fp^\times$ sa réduction modulo $p$.

Soit $(G,B,T)$ un triplet avec $G$ un groupe algébrique connexe réductif déployé sur $F$, $B \subset G$ un sous-groupe de Borel et $T \subset B$ un tore maximal déployé.
On note $B^- \subset G$ le sous-groupe de Borel opposé à $B$ par rapport à $T$ et $N$ le radical unipotent de $B$.

On note $W$ le groupe de Weyl de $(G,T)$ et $\ell : W \to \N$ la longueur relative à $B$.
On note $\Phi^+$ les racines positives de $(G,B,T)$ et $\Delta$ les racines simples de $\Phi^+$.
Pour tout $\alpha \in \Phi^+$, on note $\alpha^\vee$ la coracine correspondante et $s_\alpha \in W$ la réflexion correspondante.
Pour tout $w \in W$, on fixe un représentant $\dw \in G(F)$ dans le normalisateur de $T(F)$ et on appelle décomposition réduite toute écriture $w = s_{\alpha_1} \dots s_{\alpha_{\ell(w)}}$ avec $\alpha_1,\dots,\alpha_{\ell(w)} \in \Delta$.
On note $w_0 \in W$ l'élément de longueur maximale.

\medskip

Soit $E$ une extension finie de $\Qp$.
On note $\Oe$ l'anneau des entiers de $E$ et $\ke$ le corps résiduel de $\Oe$.
On fixe une uniformisante de $\pe \in \Oe$.
On désigne par $A$ une $\Oe$-algèbre locale artinienne de corps résiduel $\ke$ et on note encore $\omega : F^\times \to A^\times$ l'image de $\varepsilon$ dans $A^\times$.

Si $X$ est un espace topologique et $V$ est un $A$-module, on note $\Clis(X,V)$ le $A$-module constitué des fonctions $f : X \to V$ localement constantes et $\Clisc(X,V)$ le sous-$A$-module constitué des fonctions à support compact.

On emploiera la terminologie de \cite[§~2]{Em1} pour les représentations lisses à coefficients dans $A$ et on renvoie à \cite[§~3.1]{BH} pour les représentations continues unitaires admissibles sur des $E$-espaces de Banach.

\subsection*{Remerciements}

Ce travail a été réalisé sous la direction de Christophe Breuil.
Je lui suis profondément reconnaissant de m'avoir fait part de ses idées, ainsi que pour ses explications et ses remarques.
Je remercie Florian Herzig pour de nombreux commentaires qui ont permis d'améliorer cet article.

\numberwithin{theo}{subsection}

\section{Calculs de parties ordinaires dérivées} \label{sec:POD}

\subsection{Cohomologie et action de Hecke} \label{ssec:hecke}

Soient $\Nt$ un groupe algébrique unipotent sur $\Qp$ et $\Ntz \subset \Nt(\Qp)$ un sous-groupe ouvert compact standard\footnote{L'exponentielle induit un homéomorphisme $\exp : \Lie(\Nt) \iso \Nt(\Qp)$ et on dit que $\Ntz$ est \emph{standard} si $\exp^{-1}(\Ntz) \subset \Lie(\Nt)$ est une sous-$\Zp$-algèbre de Lie (voir \cite[Definition 3.5.1 et Lemma 3.5.2]{Em2}).}.
On fixe un groupe algébrique $\Lt$ sur $\Qp$ et un sous-monoïde ouvert $\Lt^+ \subset \Lt(\Qp)$.
On note $\Lt_0 \subset \Lt^+$ le sous-groupe ouvert $\Lt^+ \cap (\Lt^+)^{-1}$ de $\Lt(\Qp)$.
On suppose que $\Nt$ est muni d'une action de $\Lt$ que l'on identifie à la conjugaison dans $\Lt \ltimes \Nt$ et que $\Ntz$ est stable sous l'action par conjugaison de $\Lt^+$.
En particulier, $\Lt_0$ normalise $\Ntz$.
On note $\Lt^+ \ltimes \Ntz$ le sous-monoïde de $(\Lt \ltimes \Nt)(\Qp)$ engendré par $\Lt^+$ et $\Ntz$.
Il est ouvert car il contient le sous-groupe ouvert $\Lt_0 \ltimes \Ntz$.

Soit $V$ une représentation lisse de $\Lt^+ \ltimes \Ntz$ sur $A$.
Les $A$-modules de cohomologie $\Hc(\Ntz,V)$ calculés en utilisant des cochaînes localement constantes sont naturellement munis de l'\emph{action de Hecke} de $L^+$, définie pour tout $l \in L^+$ comme la composée
\begin{equation} \label{hecke}
\Hc(\Ntz,V) \to \Hc(l\Ntz l^{-1},V) \to \Hc(\Ntz,V)
\end{equation}
où le premier morphisme est induit par l'action de $l$ sur $V$ et le second est la corestriction de $l\Ntz l^{-1}$ à $\Ntz$.
Les foncteurs $\Hc(\Ntz,-)$ forment un $\delta$-foncteur universel de la catégorie des représentations lisses de $\Lt^+ \ltimes \Ntz$ sur $A$ dans la catégorie des représentations lisses de $\Lt^+$ sur $A$ (voir la preuve de \cite[Lemme 3.1.5]{JH}).

Soient $d$ l'entier $\dim_{\Qp} \Nt$ et $\alphat$ le caractère de la représentation adjointe de $\Lt$ sur $\detfr_{\Qp} \Lie(\Nt)$.
On a un isomorphisme $\Lt^+$-équivariant
\begin{equation} \label{dimN}
\Hc[d] (\Ntz,V) \cong V_{\Ntz} \otimes \alphat^{-1} \lvert\alphat\rvert_p^{-1}
\end{equation}
avec $V_{\Ntz}$ le $A$-module quotient des coinvariants par $\Ntz$ muni de l'action induite de $\Lt^+$ (voir la preuve de \cite[Proposition 3.1.8]{JH}).

\medskip

Soient $\Nt' \subset \Nt$ un sous-groupe fermé distingué stable sous l'action par conjugaison de $\Lt$ et $\Nt''$ le groupe quotient $\Nt/\Nt'$.
On identifie l'action induite de $\Lt$ sur $\Nt''$ à la conjugaison dans $\Lt \ltimes \Nt''$.
On note $\Ntz'$ l'intersection de $\Nt'(\Qp)$ avec $\Ntz$ et $\Ntz''$ l'image de $\Ntz$ dans $\Nt''(\Qp)$.
Ce sont des sous-groupes ouverts compacts standards de $\Nt'(\Qp)$ et $\Nt''(\Qp)$ respectivement et on a une suite exacte courte de groupes topologiques
\begin{equation*}
1 \to \Ntz' \to \Ntz \to \Ntz'' \to 1.
\end{equation*}
Les foncteurs $\Hc(\Ntz',-)$ forment un $\delta$-foncteur universel de la catégorie des représentations lisses de $\Lt^+ \ltimes \Ntz$ sur $A$ dans la catégorie des représentations lisses de $\Lt^+ \ltimes \Ntz''$ sur $A$ (voir la preuve de \cite[Lemme 3.2.1]{JH}).

\begin{lemm}
Soit $l \in \Lt^+$.
Soient $(n'_i)_{i \in \llbrack 1,k' \rrbrack},(n''_j)_{j \in \llbrack 1,k''s \rrbrack}$ des systèmes de représentants des classes à gauche $\Ntz'/l\Ntz'l^{-1},\Ntz''/l\Ntz''l^{-1}$ respectivement.
Pour tout $j \in \llbrack 1,k'' \rrbrack$, on fixe un relèvement $\nt''_j \in \Ntz$ de $n''_j \in \Ntz''$.
Alors $(\nt''_j n'_i)_{(i,j) \in \llbrack 1,k' \rrbrack \times \llbrack 1,k'' \rrbrack}$ est un système de représentants des classes à gauche $\Ntz/l\Ntz l^{-1}$.
\end{lemm}

\begin{proof}
L'application $\Ntz'/l\Ntz'l^{-1} \times \Ntz''/l\Ntz''l^{-1} \to \Ntz/l\Ntz l^{-1}$ définie par $(n'_i,n''_j) \mapsto \nt''_j n'_i$ pour tout $(i,j) \in \llbrack 1,k' \rrbrack \times \llbrack 1,k'' \rrbrack$ est bijective.

Montrons l'injectivité.
Soient $i_1,i_2 \in \llbrack 1,k' \rrbrack$ et $j_1,j_2 \in \llbrack 1,k'' \rrbrack$ tels que $\nt''_{j_1} n'_{i_1} \in \nt''_{j_2} n'_{i_2} l\Ntz l^{-1}$.
En regardant l'image dans $\Ntz''/l\Ntz''l^{-1}$, on trouve $n''_{j_1} \in n''_{j_2} l\Ntz'' l^{-1}$, d'où $j_1=j_2$.
On en déduit que $n'_{i_1} \in n'_{i_2} l\Ntz l^{-1}$, puis $n'_{i_1} \in n'_{i_2} l\Ntz' l^{-1}$ (car $\Ntz' \cap l\Ntz l^{-1} = l\Ntz'l^{-1}$), d'où $i_1=i_2$.

Montrons la surjectivité.
Soit $n \in \Ntz$.
Il existe $j \in \llbrack 1,k'' \rrbrack$ tel que l'image de $n$ dans $\Ntz''$ est dans $n''_j l\Ntz''l^{-1}$, d'où $\nt^{\prime \prime-1}_j n \in n' l\Ntz l^{-1}$ avec $n' \in \Ntz'$.
Puis, il existe $i \in \llbrack 1,k' \rrbrack$ tel que $n' \in n'_i l\Ntz'l^{-1}$, d'où $n \in \nt''_j n'_i l\Ntz l^{-1}$.
\end{proof}

En utilisant ce lemme, on voit que l'action de Hecke de $l \in \Lt^+$ sur $V^{\Ntz}$ définie par \eqref{hecke} coïncide avec celle sur $(V^{\Ntz'})^{\Ntz''}$ définie par \eqref{hecke} avec $\Ntz''$ et $V^{\Ntz'}$ au lieu de $\Ntz$ et $V$ (voir la preuve de \cite[Lemme 3.2.2]{JH}).
On en déduit l'existence d'une suite spectrale dans la catégorie des représentations lisses de $\Lt^+$ sur $A$
\begin{equation} \label{HS}
\Hc[i](\Ntz'',\Hc[j](\Ntz',V)) \Rightarrow \Hc[i+j](\Ntz,V)
\end{equation}
(voir la preuve de \cite[Proposition 3.2.3]{JH}).
Si $\dim_{\Qp} \Nt''=1$, alors en utilisant \cite[Lemma 3.5.4]{Em2} on déduit de la suite spectrale \eqref{HS} des suites exactes courtes de représentations lisses de $\Lt^+$ sur $A$
\begin{equation} \label{HSSE}
0 \to \Hc[1](\Ntz'',\Hc[n-1](\Ntz',V)) \to \Hc[n](\Ntz,V) \to \Hc[n](\Ntz',V)^{\Ntz''} \to 0
\end{equation}
pour tout entier $n>0$.
Le second morphisme non trivial est la restriction et lorsque $n=1$ le premier morphisme non trivial est l'inflation.

\medskip

On termine cette sous-section par un résultat clé.

\begin{lemm} \label{lemm:H0}
Soient $l \in \Lt^+$ un élément contractant strictement\footnote{C'est-à-dire que $\bigcap_{k \in \N} l^k \Ntz l^{-k} = 1$, ou de façon équivalente $\bigcup_{k \in \N} l^{-k} \Ntz l^k = \Nt(\Qp)$.} $\Ntz$ et $V_0$ une représentation lisse localement $l$-finie\footnote{C'est-à-dire que pour tout $v \in V_0$, le sous-$A$-module $A[l] \cdot v \subset V_0$ est de type fini.} de $\Lt(\Qp)$ sur $A$.
On suppose que l'on a une injection $\Lt^+$-équivariante $V_0 \hookrightarrow V$ et que l'action de $l$ sur son conoyau est localement nilpotente.
\begin{enumerate}
\item L'action de Hecke de $l$ sur $V^{\Ntz}$ est localement nilpotente.
\item On a une injection $\Lt^+$-équivariante $V_0 \hookrightarrow V_{\Ntz}$ et l'action de $l$ sur son conoyau est localement nilpotente.
\end{enumerate}
\end{lemm}

\begin{proof}
Soit $v \in V$.
L'action de $l$ sur $V/V_0$ étant localement nilpotente et $V_0$ étant localement $l$-finie, le sous-$A$-module $A[l] \cdot v \subset V$ est de type fini.
Par lissité de l'action de $\Ntz$ sur $V$, on en déduit que le fixateur de $A[l] \cdot v$ dans $\Ntz$ est ouvert.
Comme $l$ contracte strictement $\Ntz$, on en conclut qu'il existe $\kappa \in \N$ tel que $l^\kappa\Ntz l^{-\kappa}$ est dans le fixateur de $A[l] \cdot v$.
En particulier, on a $l^{\kappa+k} \cdot v \in V^{l^\kappa\Ntz l^{-\kappa}}$ pour tout $k \in \N$.

Montrons le point (i).
On suppose $v \in V^{\Ntz}$ et on note $\h$ l'action de Hecke.
Pour tout $k \in \N$, on a
\begin{align*}
l^{\kappa+k} \h v &= \sum_{n \in \Ntz/l^{\kappa+k}\Ntz l^{-(\kappa+k)}} n \cdot (l^{\kappa+k} \cdot v) \\
&= \left(l^\kappa\Ntz l^{-\kappa}:l^{\kappa+k}\Ntz l^{-(\kappa+k)}\right) \sum_{n \in \Ntz/l^\kappa\Ntz l^{-\kappa}} n \cdot (l^{\kappa+k} \cdot v) \\
&= \left(\Ntz:l^k\Ntz l^{-k}\right) \sum_{n \in \Ntz/l^\kappa\Ntz l^{-\kappa}} n \cdot (l^{\kappa+k} \cdot v).
\end{align*}
Or $\Ntz$ est un groupe pro-$p$ infini, $l$ contracte strictement $\Ntz$ et $A$ est artinien, donc l'indice $(\Ntz:l^k\Ntz l^{-k})$ est nul dans $A$ pour $k \in \N$ suffisamment grand.

Montrons le point (ii).
L'action de $\Lt^+$ sur $V_{\Ntz}$ étant induite par celle sur $V$, la composée $V_0 \hookrightarrow V \twoheadrightarrow V_{\Ntz}$ est $\Lt^+$-équivariante.
Comme son conoyau est un quotient de $V/V_0$, l'action de $l$ sur celui-ci est localement nilpotente.
Il suffit donc de montrer que cette composée est injective.
Pour tout $n \in \Ntz$, on a
\begin{equation*}
l^\kappa \cdot (n \cdot v - v) = (l^\kappa n l^{-\kappa}) \cdot (l^\kappa \cdot v) - (l^\kappa \cdot v) = 0.
\end{equation*}
On en déduit que l'action de $l$ sur le noyau de la surjection $V \twoheadrightarrow V_{\Ntz}$ est localement nilpotente.
Comme l'action de $l$ sur $V_0$ est inversible, on en conclut que $V_0 \cap \ker(V \twoheadrightarrow V_{\Ntz})=0$.
\end{proof}

\subsection{Filtrations de Bruhat} \label{ssec:fil}

Soient $P \subset G$ un sous-groupe parabolique standard (c'est-à-dire contenant $B$) et $L \subset P$ le sous-groupe de Levi standard (c'est-à-dire contenant $T$).
On note $N_P$ le radical unipotent de $P$.
On note $B_L \subset L$ (resp. $B_L^- \subset L$) le sous-groupe de Borel $B \cap L$ (resp. $B^- \cap L$), $N_L = N \cap L$ le radical unipotent de $B_L$, $W_L \subset W$ le groupe de Weyl de $(L,T)$, $\Phi_L^+ = \Phi^+ \cap \Phi_L$ les racines positives de $(L,B_L,T)$ et $\Delta_L = \Delta \cap \Phi_L^+$ les racines simples de $\Phi_L^+$.
On définit les représentants de Kostant des classes à gauche $W/W_L$ en posant
\begin{equation*}
\Wt_P \dfn \{ w \in W \mid \text{$w$ de longueur minimale dans $wW_L$} \}.
\end{equation*}
Pour tout $w \in W$, il existe une décomposition unique $w = \wt_P w_L$ avec $\wt_P \in \Wt_P, w_L \in W_L$ et on a $\ell(w)=\ell(\wt_P)+\ell(w_L)$ (voir \cite[Proposition 3.9]{BTC}).
La projection $W \twoheadrightarrow \Wt_P$ définie par $w \mapsto \wt_P$ respecte l'ordre de Bruhat\footnote{L'ordre de Bruhat sur $W$ est défini par $w' \leq w$ si et seulement si il existe une décomposition réduite $w=s_1 \dots s_{\ell(w)}$ et des entiers $1 \leq k_1 < \dots < k_{\ell(w')} \leq n$ tels que $w' = s_{k_1} \dots s_{k_{\ell(w')}}$.} (voir \cite[Proposition 2.5.1]{BB}).

\medskip

Soit $U$ une représentation lisse de $T(F)$ sur $A$.
On définit des filtrations de représentations induites à partir de $U$ et on calcule leurs gradués.

À partir de la décomposition de Bruhat, on obtient la décomposition $G(F) = \bigsqcup_{w \in W} (B^- \dw B)(F)$ où les relations d'adhérence entre les cellules sont données par l'ordre de Bruhat (voir \cite[§~2.1]{JH}).
En utilisant la décomposition $P(F) = B(F) \dot{W}_L B(F)$ et \cite[Lemme 3.4 (iv)]{BTC}, on en déduit les décompositions $G(F) = \bigsqcup_{\wt_P \in \Wt_P} (B^- \dwt_P P)(F)$ et $(B^- \dwt_P P)(F) = \bigsqcup_{w_L \in W_L} (B^- \dwt_P \dw_L B)(F)$ pour tout $\wt_P \in \Wt_P$, où les relations d'adhérence entre les cellules sont données par l'ordre de Bruhat.
En procédant comme dans \cite[§~2.1]{JH} (avec la notation $\cindF^C$ au lieu de $\C_C$ pour tout sous-ensemble localement fermé $B^-(F)$-invariant par translation à gauche $C \subset G(F)$), on construit une filtration naturelle de $\IndF U$ par des sous-$P(F)$-représentations $\Fil_P^\bullet \IndF U$ et pour tout $i \in \N$, on a un isomorphisme naturel $P(F)$-équivariant
\begin{equation} \label{gradP}
\Gr_P^i \IndF U \cong \bigoplus_{\ell(\wt_P)=i} \cindF^{(B^- \dwt_P P)(F)}U.
\end{equation}
De même pour tout $\wt_P \in \Wt_P$, on construit une filtration naturelle de $\cindF^{(B^- \dwt_P P)(F)} U$ par des sous-$B(F)$-représentations $\Fil_B^\bullet \cindF^{(B^- \dwt_P P)(F)} U$ et pour tout $j \in \N$, on a un isomorphisme naturel $B(F)$-équivariant
\begin{equation} \label{gradB}
\Gr_B^j \cindF^{(B^- \dwt_P P)(F)} U \cong \bigoplus_{\ell(w_L)=j} \cindF^{(B^- \dwt_P \dw_L B)(F)} U.
\end{equation}

Soit $w \in W$.
On écrit $w = \wt_P w_L$ avec $\wt_P \in \Wt_P$ et $w_L \in W_L$.
On note $U^{\wt_P}$ la représentation lisse de $T(F)$ sur $A$ dont le $A$-module sous-jacent est $U$ et sur lequel $t \in T(F)$ agit à travers $\dwt_P t \dwt{}_P^{-1}$.

On définit des sous-groupes fermés de $N$ stables sous l'action par conjugaison de $T$ en posant
\begin{gather*}
N_w \dfn N \cap (\dw^{-1} N \dw), \quad \NPw \dfn N_P \cap N_w, \\
\text{et} \quad \NLw \dfn N_L \cap (\dw_L^{-1} N_L \dw_L) = N_L \cap N_w, \label{NLw}
\end{gather*}
la dernière égalité résultant de l'égalité $\Phi_L^+ \cap w_L^{-1}(\Phi_L^+) = \Phi_L^+ \cap w^{-1}(\Phi^+)$ qui caractérise la décomposition $w=\wt_Pw_L$ (voir \cite[Proposition 3.9 (iii)]{BTC}).
Comme $N = N_L \ltimes N_P$, on a un produit semi-direct
\begin{equation} \label{psd}
N_w = \NLw \ltimes \NPw
\end{equation}
d'où un isomorphisme $A$-linéaire
\begin{equation*}
\Clisc(N_w(F),U) \cong \Clisc (\NPw(F),\Clisc(\NLw(F),U))
\end{equation*}
défini par $f \mapsto (n_P \mapsto (n_L \mapsto f(n_Ln_P)))$.
En utilisant l'isomorphisme \cite[(2)]{JH} et son analogue pour le triplet $(L,B_L,T)$ avec $U^{\wt_P}$ et $w_L$ au lieu de $U$ et $w$, on obtient un isomorphisme $A$-linéaire
\begin{equation} \label{isograd}
\cindF^{(B^- \dw B)(F)} U \cong \Clisc \left(\NPw(F),\cindFL^{(B_L^- \dw_L B_L)(F)}U^{\wt_P}\right)
\end{equation}
à travers lequel $\NPw(F)$ agit par translation à droite et l'action de $b \in (T\NLw)(F)$ sur $f \in \Clisc (\NPw(F),\cindFL^{(B_L^- \dw_L B_L)(F)}U^{\wt_P})$ est donnée par
\begin{equation*}
(b \cdot f)(n) = b \cdot f(b^{-1}nb)
\end{equation*}
pour tout $n \in \NPw(F)$.

\subsection{Calculs sur le gradué} \label{ssec:calc}

On fixe un sous-groupe ouvert compact standard $\NPz \subset N_P(F)$ et pour tout sous-groupe fermé $\Lt \subset \Res_{F/\Qp} L$, on pose
\begin{equation*}
\Lt^+ \dfn \{l \in \Lt(\Qp) \mid l\NPz l^{-1} \subset \NPz \}.
\end{equation*}
Soient $U$ une représentation lisse localement admissible de $T(F)$ sur $A$ et $w \in W$.
On écrit $w = \wt_P w_L$ avec $\wt_P \in \Wt_P$ et $w_L \in W_L$.
On calcule la cohomologie de $\NPz$ à valeurs dans
\begin{equation*}
V_w \dfn \cindF^{(B^- \dw B)(F)} U
\end{equation*}
ainsi que l'action de Hecke de $(T\NLw)^+$ sur les $A$-modules $\Hc(\NPz,V_w)$.

On fixe une décomposition réduite $\wt_P = s_{\ell(\wt_P)} \dots s_1$.
Soit $k \in \llbrack 0,\ell(\wt_P) \rrbrack$.
On a $s_k \dots s_1 \in \Wt_P$ et $\NPs{k}$ est stable sous l'action par conjugaison de $T\NLw$ (voir le produit semi-direct \eqref{psd} avec $s_k \dots s_1$ au lieu de $\wt_P$).
On note $\NPsz{k}$ l'intersection de $\NPs{k}(F)$ avec $\NPz$.
Si $k<\ell(\wt_P)$, alors $\NPs{k+1}$ est de codimension $1$ dans $\NPs{k}$ qui est nilpotent, donc il est distingué d'après \cite[Chapitre IV, §~4, Corollaire 1.9]{DG}.
Dans ce cas, on pose
\begin{equation*}
\NPs{k}'' \dfn \NPs{k}/\NPs{k+1}
\end{equation*}
et on note $\NPsz{k}''$ l'image de $\NPsz{k}$ dans $\NPs{k}''(F)$.
La suite spectrale \eqref{HS} avec $\Lt^+=(T\NLw)^+$, $\Ntz=\NPsz{k}$, $\Ntz'=\NPsz{k+1}$, $\Ntz''=\NPsz{k}''$ et $V=V_w$ est une suite spectrale de représentations lisses de $(T\NLw)^+$ sur $A$
\begin{equation} \label{HSk}
\Hc[i](\NPsz{k}'',\Hc[j](\NPsz{k+1},V_w)) \Rightarrow \Hc[i+j](\NPsz{k},V_w)
\end{equation}
et si $i=[F:\Qp]$, alors l'isomorphisme \eqref{dimN} avec $\Lt^+=(T\NLw)^+$, $\Ntz=\NPsz{k}''$ et $V=\Hc[j](\NPsz{k+1},V_w)$ est un isomorphisme $(T\NLw)^+$-équivariant
\begin{multline} \label{dimNk}
\Hc[i](\NPsz{k}'',\Hc[j](\NPsz{k+1},V_w)) \\
\cong \Hc[j](\NPsz{k+1},V_w)_{\NPsz{k}''} \otimes (\oma)
\end{multline}
où $\alpha$ est le caractère algébrique (trivial sur $\NLw$) de la représentation adjointe de $T\NLw$ sur $\Lie(\NPs{k}'')$.

\begin{lemm} \label{lemm:H1}
Soient $k \in \llbrack 0,\ell(\wt_P) \rrbrack$ et $n \in \N$.
Si $n>[F:\Qp] \cdot (\ell(\wt_P)-k)$, alors $\Hc[n](\NPsz{k},V_w)=0$.
\end{lemm}

\begin{proof}
On procède par récurrence décroissante sur $k$ comme dans la preuve de \cite[Lemme 3.3.2]{JH}.
Pour l'initialisation, on montre que $V_w$ est $\NPwz$-acyclique en utilisant l'isomorphisme $\NPwz$-équivariant
\begin{equation*}
V_w \cong \bigoplus_{n \in \NPw(F)/\NPwz} \Clis \left(n \NPwz,\cindFL^{(B_L^- \dw_L B_L)(F)} U^{\wt_P}\right)
\end{equation*}
(qui se déduit de l'isomorphisme \eqref{isograd}) et le fait que la cohomologie de $\NPwz$ commute aux sommes directes (car l'image d'une cochaîne continue est finie par compacité de $\NPwz$).
Pour l'itération, on utilise la suite spectrale \eqref{HSk} avec \cite[Lemma 3.5.4]{Em2}.
\end{proof}

Soit $S$ le plus grand sous-tore déployé de $\Res_{F/\Qp} T$.
On fixe un élément $\zeta \in S^+ \cap Z_L^+$ contractant strictement $\NPz$ (par exemple $\zeta=\lambda(p)$ avec $\lambda$ un cocaractère algébrique de $T$ associé à $P$, c'est-à-dire tel que $\langle \alpha,\lambda \rangle \geq 0$ pour tout $\alpha \in \Phi^+$ avec égalité si et seulement si $\alpha \in \Phi_L^+$).

On note $\alpha_k$ le caractère algébrique (trivial sur $\NLw$) de la représentation adjointe de $T\NLw$ sur $\detfr_F \Lie(\Ns{k} \cap N_{w_0\wt_P})$ et on définit une représentation lisse de $B_L(F)$ sur $A$ en posant
\begin{equation*}
V_k \dfn \cindFL^{(B_L^- \dw_L B_L)(F)} \left(U^{\wt_P} \otimes (\oma_k)\right).
\end{equation*}
D'après \cite[Proposition 4.1.7]{Em1}, $\IndFL (U^{\wt_P} \otimes (\oma_k))$ est localement admissible.
En utilisant \cite[Lemma 2.3.4]{Em1}, on en déduit que le sous-quotient $V_k$ est localement $Z_L(F)$-finie, donc localement $\zeta$-finie.

\begin{lemm} \label{lemm:H2}
Soient $k \in \llbrack 0,\ell(\wt_P) \rrbrack$ et $n \in \N$.
Si $n = [F:\Qp] \cdot (\ell(\wt_P)-k)$, alors on a une injection $(T\NLw)^+$-équivariante $V_k \hookrightarrow \Hc[n](\NPsz{k},V_w)$ et l'action de Hecke de $\zeta$ sur son conoyau est localement nilpotente.
\end{lemm}

\begin{proof}
On procède par récurrence décroissante sur $k$ comme dans la preuve de \cite[Lemme 3.3.3]{JH}.
Pour l'initialisation, on note $\Vwz \subset V_w$ le sous-$A$-module constitué des fonctions à support dans $\NPwz$ à travers l'isomorphisme \eqref{isograd} et on montre que l'évaluation en $1 \in \NPw(F)$ induit un isomorphisme $(T\NLw)^+$-équivariant $\Vwz^{\NPwz} \iso V_0$, d'où une injection $(T\NLw)^+$-équivariante $V_0 \hookrightarrow V_w^{\NPwz}$ et l'action de Hecke de $\zeta$ sur son conoyau est localement nilpotente (car l'action de $\zeta$ sur $V_w/\Vwz$ est localement nilpotente, voir la preuve de \cite[Lemme 3.3.1]{JH}).
Pour l'itération, on utilise la suite spectrale \eqref{HSk} avec \cite[Lemma 3.5.4]{Em2}, le lemme \ref{lemm:H1} et l'isomorphisme \eqref{dimNk}, ainsi que le point (ii) du lemme \ref{lemm:H0} : on obtient une injection $(T\NLw)^+$-équivariante $V_{k+1} \otimes (\oma) \hookrightarrow \Hc[n](\NPsz{k},V_w)$ et l'action de Hecke de $\zeta$ sur son conoyau est localement nilpotente ; puis, on utilise l'isomorphisme naturel $(T\NLw)^+$-équivariant $V_{k+1} \otimes (\oma) \cong V_k$ qui résulte de l'isomorphisme naturel $(T\NLw)^+$-équivariant
\begin{multline*}
\cindFL^{(B_L^- \dw_L B_L)(F)} \left( U^{\wt_P} \otimes (\oma_{k+1}) \right) \otimes (\oma) \\
\cong \cindFL^{(B_L^- \dw_L B_L)(F)} \left( U^{\wt_P} \otimes (\omega^{-1} \circ (\alpha_{k+1} + w_L(\alpha))) \right)
\end{multline*}
et du fait que $w_L(\alpha)$ est le caractère algébrique de la représentation adjointe de $T\NLw$ sur $\Lie(\Ns{k}/\Ns{k+1})$, d'où $\alpha_{k+1} + w_L(\alpha) = \alpha_k$.
\end{proof}

\begin{lemm} \label{lemm:H3}
Soient $k \in \llbrack 0,\ell(\wt_P) \rrbrack$ et $n \in \N$.
Si $n<[F:\Qp] \cdot (\ell(\wt_P)-k)$, alors l'action de Hecke de $\zeta$ sur $\Hc[n](\NPsz{k},V_w)$ est localement nilpotente.
\end{lemm}

\begin{proof}
On procède par récurrence décroissante sur $k$ comme dans la preuve de \cite[Lemme 3.3.4]{JH}, en utilisant en utilisant la suite spectrale \eqref{HSk} avec le lemme \ref{lemm:H1}.
Lorsque $j = [F:\Qp] \cdot (\ell(\wt_P)-(k+1))$ et $i < [F:\Qp]$, on utilise le lemme \ref{lemm:H2} : en notant $V$ la représentation lisse $\Hc[j](\NPsz{k+1},V_w)$ de $S^+ \ltimes \NPsz{k}''$ sur $A$, on a une injection $S^+$-équivariante $V_k \hookrightarrow V$ telle que l'action de $\zeta$ sur son conoyau est localement nilpotente.
Pour conclure, on montre que l'action de Hecke de $\zeta$ sur $\Hc[i](\NPsz{k}'',V)$ est localement nilpotente.

Comme $\Res_{F/\Qp} \NPs{k}''$ est nilpotent et commutatif, l'exponentielle est un isomorphisme de groupes (voir \cite[Chapitre IV, §~2, Proposition 4.1]{DG}).
De plus, l'action adjointe de $S$ sur $\Lie(\Res_{F/\Qp} \NPs{k}'')$ se factorise à travers un caractère algébrique $\alphat$.
On en déduit qu'il existe une suite de sous-groupes fermés stables sous l'action par conjugaison de $S^+$
\begin{equation*}
1 = \Nt_0 \subset \Nt_1 \subset \dots \subset \Nt_{[F:\Qp]} = \Res_{F/\Qp} \NPs{k}''
\end{equation*}
dont les quotients successifs sont isomorphes au groupe additif sur $\Qp$.
Soit $l \in \llbrack 0,[F:\Qp] \rrbrack$.
On note $\Ntlz$ l'intersection de $\Nt_l(\Qp)$ avec $\NPsz{k}''$.
Si $l>0$, alors on note $\Ntlz''$ l'image de $\Ntlz$ dans $(\Nt_l/\Nt_{l-1})(\Qp)$.
On montre par récurrence sur $l$ les points suivants.
\begin{enumerate}
\item Si $i=l$, alors on a une injection $S^+$-équivariante $V_k \otimes \alphat^{-i} \lvert\alphat\rvert_p^{-i} \hookrightarrow \Hc[i](\Ntlz,V)$ et l'action de Hecke de $\zeta$ sur son conoyau est localement nilpotente.
\item Si $i<l$, alors l'action de Hecke de $\zeta$ sur $\Hc[i](\Ntlz,V)$ est localement nilpotente.
\end{enumerate}

Le cas $l=0$ est vrai par hypothèse.
On suppose $l>0$ et le résultat vrai pour $l-1$.
En utilisant la suite exacte \eqref{HSSE} avec $\Lt^+=S^+$, $\Ntz=\Ntlz$, $\Ntz'=\Ntlz[l-1]$ et $\Ntz''=\Ntlz''$ et l'isomorphisme \eqref{dimN} avec $\Lt^+=S^+$, $\Ntz=\Ntlz''$ et $\Hc[i-1](\Ntlz[l-1],V)$ au lieu de $V$, on obtient une suite exacte courte de représentations lisses de $S^+$ sur $A$
\begin{equation} \label{SEL}
0 \to \Hc[i-1](\Ntlz[l-1],V)_{\Ntlz''} \otimes \alphat^{-1}\lvert\alphat\rvert_p^{-1} \to \Hc[i](\Ntlz,V) 
\to \Hc[i](\Ntlz[l-1],V)^{\Ntlz''} \to 0.
\end{equation}

On suppose $i=l$ et on prouve le point (i).
D'un côté $i-1=l-1$ et par l'hypothèse de récurrence, on a une injection $S^+$-équivariante $V_k \otimes \alphat^{-(i-1)} \lvert\alphat\rvert_p^{-(i-1)} \hookrightarrow \Hc[i-1](\Ntlz[l-1],V)$ et l'action de Hecke de $\zeta$ sur son conoyau est localement nilpotente, donc d'après le point (ii) du lemme \ref{lemm:H0} avec $\Lt^+=S^+$, $\Ntz=\Ntlz''$, $\Hc[i-1](\Ntlz[l-1],V)$ et $V_k \otimes \alphat^{-(i-1)} \lvert\alphat\rvert_p^{-(i-1)}$ au lieu de $V$ et $V_k$, on a une injection $S^+$-équivariante $V_k \otimes \alphat^{-(i-1)} \lvert\alphat\rvert_p^{-(i-1)} \hookrightarrow \Hc[i-1](\Ntlz[l-1],V)_{\Ntlz''}$ et l'action de Hecke de $\zeta$ sur son conoyau est localement nilpotente.
De l'autre $i>l-1$, donc $\Hc[i](\Ntlz[l-1],V)=0$ d'après \cite[Lemma 3.5.4]{Em2}, d'où un isomorphisme $S^+$-équivariant $\Hc[i-1](\Ntlz[l-1],V)_{\Ntlz''} \otimes \alphat^{-1} \lvert\alphat\rvert_p^{-1} \iso \Hc[i](\Ntlz,V)$.
En utilisant la suite exacte \eqref{SEL}, on en déduit le point (i).

On suppose $i<l$ et on prouve le point (ii).
D'un côté $i-1<l-1$, donc l'action de Hecke de $\zeta$ sur $\Hc[i-1](\Ntlz[l-1],V)$ est localement nilpotente par hypothèse de récurrence.
De l'autre ou bien $i<l-1$ et l'action de Hecke de $\zeta$ sur $\Hc[i](\Ntlz[l-1],V)$ est localement nilpotente par hypothèse de récurrence, donc l'action de Hecke de $\zeta$ sur $\Hc[i](\Ntlz[l-1],V)^{\Ntlz''}$ est localement nilpotente ; ou bien $i=l-1$ et par hypothèse de récurrence on a une injection $S^+$-équivariante $V_k \otimes \alphat^{-i} \lvert\alphat\rvert_p^{-i} \hookrightarrow \Hc[i](\Ntlz[l-1],V)$ et l'action de Hecke de $\zeta$ sur son conoyau est localement nilpotente, donc l'action de Hecke de $\zeta$ sur $\Hc[i](\Ntlz[l-1],V)^{\Ntlz''}$ est localement nilpotente d'après le point (i) du lemme \ref{lemm:H0} avec $\Lt^+=S^+$, $\Ntz=\Ntlz''$, $\Hc[i-1](\Ntlz[l-1],V)$ et $V_k \otimes \alphat^{-i} \lvert\alphat\rvert_p^{-i}$ au lieu de $V$ et $V_k$.
En utilisant la suite exacte \eqref{SEL}, on en déduit le point (ii).
\end{proof}

\subsection{Calculs sur une induite} \label{ssec:HnOrd}

On rappelle que si $V$ est une représentation lisse $V$ de $P(F)$ sur $A$, alors ses \emph{parties ordinaires dérivées} sont définies par
\begin{equation*} 
\HOrdF V \dfn \Hom_{A[Z_L^+]} \left( A[Z_L(F)],\Hc(\NPz,V) \right)_{Z_L(F)-\fin}
\end{equation*}
(voir \cite[Definition 3.3.1]{Em2}).
Pour tout sous-groupe fermé $\Pt \subset P$ tel que $\Pt=\Lt N_P$ avec $\Lt \subset L$ un sous-groupe fermé contenant $Z_L$, si $V$ est une représentation lisse de $\Pt(F)$ sur $A$, alors les $A$-modules $\HOrdF V$ sont naturellement des représentations lisses de $\Lt(F)$ sur $A$ (car le produit induit un isomorphisme de groupes $\Lt^+ \times_{Z_L^+} Z_L(F) \iso \Lt(F)$, voir la preuve de \cite[Proposition 3.3.6]{Eme06}).
En particulier avec $\Pt=B$, on a $\Lt=B_L$.

\medskip

Soient $U$ une représentation lisse localement admissible de $T(F)$ sur $A$ et $n \in \N$.
On montre que les filtrations de Bruhat induisent des filtrations des parties ordinaires dérivées et on calcule partiellement ces dernières.

D'après \cite[Theorem 3.4.7]{Em2}, $\Hc[n](\NPz,\IndF U)$ est réunion de sous-$A$-modules de type fini stables par $Z_L^+$.
En procédant comme dans \cite[§~2.2]{JH}, on voit que la filtration $\Fil_P^\bullet \IndF U$ induit une filtration de $\Hc[n](\NPz,\IndF U)$ par des sous-$L(F)$-représentations.
En particulier, ces dernières sont réunions de sous-$A$-modules de type fini stables par $Z_L^+$.
En utilisant \cite[Lemma 3.2.1 (3)]{Em2}, on en déduit que la filtration $\Fil_P^\bullet \IndF U$ induit une filtration naturelle de $\HOrdF[n](\IndF U)$ par des sous-$L(F)$-représentations $\Fil_P^\bullet \HOrdF[n](\IndF U)$ et que pour tout $i \in \N$, l'isomorphisme \eqref{gradP} induit un isomorphisme naturel $L(F)$-équivariant
\begin{equation} \label{HOrdgradP}
\Gr_P^i \HOrdF[n] \left( \IndF U \right) \cong \bigoplus_{\ell(\wt_P)=i} \HOrdF[n] \left( \cindF^{(B^- \dwt_P P)(F)} U \right).
\end{equation}
De même pour tout $\wt_P \in \Wt_P$, la filtration $\Fil_B^\bullet \cindF^{(B^- \dwt_P P)(F)} U$ induit une filtration naturelle de $\HOrdF[n](\cindF^{(B^- \dwt_P P)(F)} U)$ par des sous-$B_L(F)$-représentations $\Fil_B^\bullet \HOrdF[n](\cindF^{(B^- \dwt_P P)(F)} U)$ et pour tout $j \in \N$, l'isomorphisme \eqref{gradB} induit un isomorphisme naturel $B_L(F)$-équivariant
\begin{multline} \label{HOrdgradB}
\Gr_B^j \HOrdF[n] \left( \cindF^{(B^- \dwt_P P)(F)} U \right) \\
\cong \bigoplus_{\ell(w_L)=j} \HOrdF[n] \left( \cindF^{(B^- \dwt_P \dw_L B)(F)} U \right).
\end{multline}

Par ailleurs, en procédant comme dans la sous-section \ref{ssec:fil} pour le triplet $(L,B_L,T)$ avec $\Ut$ une représentation lisse de $T(F)$ sur $A$, on construit une filtration naturelle de $\IndFL \Ut$ par des sous-$B_L(F)$-représentations $\Fil_{B_L}^\bullet \IndFL \Ut$ et pour tout $j \in \N$, on a un isomorphisme $B_L(F)$-équivariant
\begin{equation} \label{gradBL}
\Gr_{B_L}^j \IndFL \Ut \cong \bigoplus_{\ell(w_L)=j} \cindFL^{(B_L^- \dw_L B_L)(F)} \Ut.
\end{equation}

Pour tout $\wt_P \in \Wt_P$, on rappelle que $U^{\wt_P}$ est la représentation lisse de $T(F)$ sur $A$ dont le $A$-module sous-jacent est $U$ et sur lequel $t \in T(F)$ agit à travers $\dwt_P t \dwt{}_P^{-1}$ et on note $\alpha_{\wt_P}$ le caractère algébrique de la représentation adjointe de $T$ sur $\detfr_F \Lie(N_{w_0\wt_P})$.
On a $\alpha_1=1$ et $\alpha_{s_\alpha}=\alpha$ pour tout $\alpha \in \Delta-\Delta_L$.

\begin{theo} \label{theo:HnOrd}
Soient $U$ une représentation lisse localement admissible de $T(F)$ sur $A$ et $n \in \N$.
On a un isomorphisme naturel $L(F)$-équivariant
\begin{equation*}
\HOrdF[n] \left( \IndF U \right) \cong \bigoplus_{[F:\Qp] \cdot \ell(\wt_P)=n} \HOrdF[n] \left( \cindF^{(B^- \dwt_P P)(F)} U \right).
\end{equation*}
Soit $\wt_P \in \Wt_P$ tel que $[F:\Qp] \cdot \ell(\wt_P) = n$.
Pour tout $j \in \N$, on a un isomorphisme naturel $T(F)$-équivariant
\begin{equation*}
\Gr_B^j \HOrdF[n] \left( \cindF^{(B^- \dwt_P P)(F)} U \right) \cong \Gr_{B_L}^j \IndFL \left( U^{\wt_P} \otimes (\oma_{\wt_P}) \right)
\end{equation*}
dont la restriction au facteur direct $\cindFL^{(B_L^- \dw_L B_L)(F)} (U^{\wt_P} \otimes (\oma_{\wt_P}))$ à travers l'isomorphisme \eqref{gradBL} avec $\Ut=U^{\wt_P} \otimes (\oma_{\wt_P})$ est $\NLw(F)$-équivariante pour tout $w_L \in W_L$ tel que $\ell(w_L)=j$.
\end{theo}

\begin{proof}
Pour tout $\wt_P \in \Wt_P$, on déduit des lemmes \ref{lemm:H1}, \ref{lemm:H2} et \ref{lemm:H3} avec $k=0$ un isomorphisme $(T\NLw)(F)$-équivariant
\begin{equation} \label{iso}
\HOrdF[n] \left( \cindF^{(B^- \dw B)(F)} U \right) \cong \cindFL^{(B_L^- \dw_L B_L)(F)} \left( U^{\wt_P} \otimes (\oma_{\wt_P}) \right)
\end{equation}
si $[F:\Qp] \cdot \ell(\wt_P) = n$ et l'égalité $\HOrdF[n](\cindF^{(B^- \dw B)(F)} U)=0$ sinon ; en utilisant l'isomorphisme \eqref{HOrdgradP}, on obtient le second isomorphisme de l'énoncé si $[F:\Qp] \cdot \ell(\wt_P) = n$ et l'égalité $\HOrdF[n](\cindF^{(B^- \dwt_P P)(F)} U)=0$ sinon.
En utilisant l'isomorphisme \eqref{HOrdgradB}, on déduit de ces égalités que le gradué $\Gr_P^\bullet \HOrdF[n] (\IndF U)$ est nul si $[F:\Qp] \nmid n$ et concentré en degré $n/[F:\Qp]$ sinon.
Ainsi, la filtration $\Fil_P^\bullet \HOrdF[n](\IndF U)$ est triviale et on obtient le premier isomorphisme de l'énoncé.

La naturalité de l'isomorphisme \eqref{iso} est une conséquence de la naturalité des filtrations de Bruhat, de l'inclusion $\Vwz \subset V_w$ qui induit l'injection du lemme \ref{lemm:H2}, de la suite spectrale \eqref{HSk} et de l'isomorphisme \eqref{dimNk}.
\end{proof}

\begin{rema}
On s'attend à ce que l'isomorphisme \eqref{iso} soit en fait $B_L(F)$-équivariant.
Les calculs sont limités par le fait que les dévissages de $N_P$ utilisés dans la sous-section \ref{ssec:calc} ne sont pas stables sous l'action par conjugaison de $B_L$.
\end{rema}

\begin{coro} \label{coro:HnOrdF}
Soient $U$ une représentation lisse localement admissible de $T(F)$ sur $A$ et $n \in \N$.
Si $[F:\Qp] \nmid n$, alors $\HOrdF[n] (\IndF U) = 0$.
\end{coro}

Soit $\wt_P \in \Wt_P$ tel que $[F:\Qp] \cdot \ell(\wt_P) = n$.
On déduit du théorème \ref{theo:HnOrd} une injection naturelle $B_L(F)$-équivariante
\begin{equation*}
\Fil_{B_L}^0 \IndFL \left( U^{\wt_P} \otimes (\oma_{\wt_P}) \right) \hookrightarrow \HOrdF[n] \left( \cindF^{(B^- \dwt_P P)(F)} U \right)
\end{equation*}
qui se prolonge naturellement en un morphisme $L(F)$-équivariant
\begin{equation} \label{HnOrd}
\IndFL \left( U^{\wt_P} \otimes (\oma_{\wt_P}) \right) \to \HOrdF[n] \left( \cindF^{(B^- \dwt_P P)(F)} U \right)
\end{equation}
(voir la preuve de \cite[Theorem 4.4.6]{Em1}).

\begin{conj} \label{conj:HnOrd}
Le morphisme naturel \eqref{HnOrd} est un isomorphisme.
\end{conj}

Lorsque $n=0$, la conjecture \ref{conj:HnOrd} est vraie : d'après \cite[Proposition 4.3.4]{Em1}, on a un isomorphisme naturel $L(F)$-équivariant
\begin{equation} \label{Ord}
\OrdF \left( \IndF U \right) \cong \IndFL U.
\end{equation}
On prouve aussi la conjecture lorsque la source du morphisme est irréductible.

\begin{prop} \label{prop:HnOrdirr}
Si la représentation $\IndFL (U^{\wt_P} \otimes (\oma_{\wt_P}))$ est irréductible, alors le morphisme naturel \eqref{HnOrd} est un isomorphisme.
\end{prop}

\begin{proof}
On suppose la représentation $\IndFL (U^{\wt_P} \otimes (\oma_{\wt_P}))$ irréductible, donc $U$ irréductible.
Comme le morphisme \eqref{HnOrd} est non nul, on en déduit qu'il est injectif.
Il reste à montrer sa surjectivité.
Par extension des scalaires on se ramène au cas où $U$ est absolument irréductible (voir \cite[Lemma 4.1.9]{Em2}), donc de dimension $1$ sur $\ke$ (car $T(F)$ est commutatif).
Dans ce cas, on montre que la source et le but du morphisme \eqref{HnOrd} sont de longueur finie égale à $\card W_L$ dans la catégorie des représentations lisses de $B_L(F)$ sur $\ke$ : pour tout $w_L \in W_L$, on a un isomorphisme $(T\NLw)(F)$-équivariant
\begin{multline*}
\cindFL^{(B_L^- \dw_L B_L)(F)} \left( U^{\wt_P} \otimes (\oma_{\wt_P}) \right) \\
\cong \Clisc (\NLw(F),\ke) \otimes_{\ke} \left( U^{\wt_P} \otimes (\oma_{\wt_P}) \right)^{w_L}
\end{multline*}
et $\Clisc (\NLw(F),\ke)$ est irréductible dans la catégorie des représentations lisses de $(T\NLw)(F)$ sur $\ke$ d'après \cite[Théorème 5]{Vig}, donc en utilisant le théorème \ref{theo:HnOrd} on en déduit le résultat.
\end{proof}

\section{Application aux extensions}

\subsection{Caractères génériques}

Soit $\chi$ un caractère de $T(F)$ à valeurs dans le groupe des unités d'un anneau quelconque.
Pour tout $w \in W$, on note $w(\chi)$ le caractère de $T(F)$ défini par $w(\chi)(t) = \chi(\dw^{-1} t \dw)$ pour tout $t \in T(F)$.

\begin{defi} \label{defi:gen}
On dit que $\chi$ est :
\begin{itemize}
\item \emph{faiblement générique} si $s_\alpha(\chi) \neq \chi$ pour tout $\alpha \in \Delta$ ;
\item \emph{générique} si $s_\alpha(\chi) \neq \chi$ pour tout $\alpha \in \Phi^+$ ;
\item \emph{fortement générique} si $w(\chi) \neq \chi$ pour tout $w \in W-\{1\}$.
\end{itemize}
\end{defi}

\begin{rema} \phantomsection \label{rema:gen}
\begin{enumerate}
\item Soit $\alpha \in \Phi^+$.
Si $s_\alpha(\chi) \neq \chi$, alors $\chi \circ \alpha^\vee \neq 1$ ; la réciproque est vraie lorsque le centre de $G$ est connexe (voir \cite[Lemme 5.1.2]{JH}), mais pas en général (voir \cite[Lemme 3.1.2]{JHC}).
\item Si $\chi$ est générique, alors $w(\chi)$ est générique pour tout $w \in W$.
\item Si $\chi$ est générique, alors $\chi$ est fortement générique lorsque $G=\GL_n$, mais pas en général (voir l'exemple \ref{exem:gen} ci-dessous).
\end{enumerate}
\end{rema}

\begin{exem} \label{exem:gen}
On suppose $F=\Qp$ et $G=\Gd$.
On note $\Delta=\{\alpha,\beta\}$.
Soit $\chi : T(\Qp) \to \{\pm 1\}$ le caractère défini par
\begin{equation*}
\chi(t) = (-1)^{\val_p(\alpha(t))} (-1)^{\left(\frac{\omega(\beta(t))}p\right)}
\end{equation*}
avec $\val_p : \Qp^\times \to \Z$ la valuation $p$-adique et $(\frac{}p) : \Fp^\times \to \{\pm1\}$ le symbole de Legendre (c'est-à-dire le résidu quadratique modulo $p$).
Si $p \neq 2$, alors $\chi$ est générique (car $\Gd$ est semi-simple, simplement connexe et de centre connexe, donc $\langle \alpha,\gamma^\vee \rangle = 1$ ou $\langle \beta,\gamma^\vee \rangle = 1$ pour tout $\gamma \in \Phi^+$) mais $w_0(\chi)=\chi$ (car $w_0(t)=t^{-1}$ pour tout $t \in T(\Qp)$).
On note que $w_0 = s_\alpha s_\beta s_\alpha s_\beta s_\alpha s_\beta$.
\end{exem}

\begin{lemm} \label{lemm:gen}
On suppose que le centre de $G$ est connexe.
Soient $w \in W$ et $s_1,\dots,s_n \in W$ des réflexions simples telles que $w=s_n \dots s_1$.
\begin{enumerate}
\item Soit $\alpha \in \Delta$.
Si $\chi \circ \alpha^\vee \neq 1$, $s_1=s_\alpha$ et $s_k \neq s_\alpha$ pour tout $k \in \llbrack 2,n \rrbrack$, alors $w(\chi) \neq \chi$.
\item Si $\chi$ est générique et s'il existe $k_0 \in \llbrack 1,n \rrbrack$ tel que $s_k \neq s_{k_0}$ pour tout $k \in \llbrack 1,n \rrbrack -\{k_0\}$, alors $w(\chi) \neq \chi$.
\item Si $\chi$ est générique et $\ell(w) \leq 5$, alors $w(\chi) \neq \chi$.
\end{enumerate}
\end{lemm}

\begin{rema}
Dans le point (iii), la condition $\ell(w) \leq 5$ est optimale : il existe des contre-exemples lorsque $\ell(w)=6$ (voir l'exemple \ref{exem:gen}).
\end{rema}

\begin{proof}
Pour tout cocaractère algébrique $\lambda$ de $T$, on a
\begin{equation*}
(\chi \cdot w(\chi)^{-1}) \circ \lambda = \chi \circ (\lambda-w^{-1}(\lambda)) = \chi \circ \big( \sum_{k=1}^n s_1 \dots s_{k-1} (\lambda - s_k(\lambda)) \big).
\end{equation*}
Soient $\alpha \in \Delta$ et $\lambda_\alpha$ un copoids fondamental correspondant à $\alpha$ (voir \cite[Proposition 2.1.1]{BH}).
Pour tout $k \in \llbrack 1,n \rrbrack$, on a
\begin{equation*}
\lambda_\alpha - s_k(\lambda_\alpha) =
\begin{cases}
\alpha^\vee &\text{si $s_k=s_\alpha$,} \\
0& \text{sinon.}
\end{cases}
\end{equation*}
On suppose $s_\alpha = s_{k_0}$ avec $k_0 \in \llbrack 1,n \rrbrack$ et $s_\alpha \neq s_k$ pour tout $k \in \llbrack 1,n \rrbrack -\{k_0\}$.
Si $k_0=1$ et $\chi \circ \alpha^\vee \neq 1$, ou encore si $\chi$ est générique, alors
\begin{equation*}
(\chi \cdot w(\chi)^{-1}) \circ \lambda_\alpha = \chi \circ s_1 \dots s_{k_0-1} (\alpha)^\vee \neq 1.
\end{equation*}

On suppose $\chi$ générique et $\ell(w) \leq 5$.
Si $\ell(w) \leq 3$, alors la condition du point (ii) est automatiquement vérifiée par une décomposition réduite de $w$, donc $w(\chi) \neq \chi$.
Si $\ell(w) = 5$ et la condition du point (ii) n'est pas vérifiée pour une décomposition réduite de $w$, alors nécessairement cette décomposition réduite est de la forme $w = s_\alpha s_\beta s_\alpha s_\beta s_\alpha$ avec $\alpha, \beta \in \Delta$ et d'après le point (ii), on a $s_\beta s_\alpha s_\beta(s_\alpha(\chi)) \neq s_\alpha(\chi)$ par généricité de $s_\alpha(\chi)$, d'où $w(\chi) \neq \chi$.
On suppose que $\ell(w)=4$ et que la condition du point (ii) n'est pas vérifiée par une décomposition réduite de $w$.
Nécessairement, cette décomposition réduite est de la forme $w= s_\beta s_\alpha s_\beta s_\alpha$ avec $\alpha,\beta \in \Delta$ distinctes non orthogonales.
Dans ce cas, on déduit de \cite[Chapitre IV, §~1.3]{BbkLIE46} que $\langle \alpha,\beta^\vee \rangle = -1$ ou $\langle \beta,\alpha^\vee \rangle = -1$ (le signe résulte du fait que $\alpha$ et $\beta$ sont simples).
Quitte à remplacer $w$ par $w^{-1}$, on peut supposer $\langle \beta,\alpha^\vee \rangle = -1$.
Dans ce cas, on a
\begin{equation*}
(\chi \cdot w(\chi)^{-1}) \circ \lambda_\alpha = \chi \circ \left( \alpha^\vee + s_\alpha s_\beta(\alpha)^\vee \right) = \chi \circ s_\alpha(\beta)^\vee \neq 1. \qedhere
\end{equation*}
\end{proof}

\subsection{Inexistence de chaînes de séries principales}

On suppose $F=\Qp$, le centre de $G$ connexe et le groupe dérivé de $G$ simplement connexe.
On note $\theta$ la somme des poids fondamentaux (bien définie à un caractère algébrique de $G$ près, voir \cite[Proposition 2.1.1]{BH}).

On montre que certaines \og chaînes \fg{} de trois séries principales n'existent pas dans la catégorie des représentations continues unitaires admissibles de $G(\Qp)$ sur $E$.
Le résultat analogue modulo $p$ (c'est-à-dire dans la catégorie des représentations lisses admissibles de $G(\Qp)$ sur $\ke$) se démontre de façon analogue.

\begin{theo} \label{theo:chaine}
Soient $\chi, \chi', \chi'' : T(\Qp) \to \Oe^\times \subset E^\times$ des caractères continus unitaires.
On suppose $\chi \neq \chi''$ et si $\chi'=\chi$ ou $\chi'=\chi''$, alors $\chi'$ faiblement générique.
Alors il n'existe pas de représentation continue unitaire admissible de $G(\Qp)$ sur $E$ ayant $\IndQp \chi \epsth$ pour socle, $\IndQp \chi'' \epsth$ pour cosocle et $\IndQp \chi' \epsth$ pour unique constituant irréductible intermédiaire.
\end{theo}

\begin{rema}
Si $\chi''=\chi$ et $\chi'=s_\alpha(\chi) \neq \chi$ avec $\alpha \in \Delta$, alors il existe une unique chaîne comme dans le théorème à isomorphisme près et elle est obtenue par induction parabolique à partir de $\GL_2(\Qp)$.
\end{rema}

\begin{proof}
On suppose que les séries principales de l'énoncé sont topologiquement irréductibles (sinon le résultat est trivial) et qu'il existe des extensions non scindées de $\IndQp \chi' \epsth$ par $\IndQp \chi \epsth$ et de $\IndQp \chi'' \epsth$ par $\IndQp \chi' \epsth$ dans la catégorie des représentations continues unitaires admissibles de $G(\Qp)$ sur $E$ (sinon il n'existe pas de chaîne comme dans l'énoncé).

\medskip

\emph{1\ier{} cas : $\chi' \neq \chi$ et $\chi' \neq \chi''$.}

Dans ce cas, on déduit de \cite[Théorème 1.1 (i)]{JH} que $\chi' = s_\beta(\chi)$ et $\chi'' = s_\alpha s_\beta(\chi)$ avec $\alpha, \beta \in \Delta$ distinctes telles que $\chi \circ \beta^\vee \neq 1$ et $\chi \circ s_\beta(\alpha)^\vee \neq 1$ (voir le point (i) de la remarque \ref{rema:gen}).

On note $G_\alpha \subset G$ le sous-groupe fermé engendré par $T$ et les sous-groupes radiciels correspondant aux racines $\pm\alpha$, $P_\alpha \subset G$ (resp. $P_\alpha^- \subset G$) le sous-groupe parabolique standard $BG_\alpha$ (resp. $B^-G_\alpha$) et $B_\alpha \subset G_\alpha$ (resp. $B_\alpha^- \subset G_\alpha$) le sous-groupe de Borel $B \cap G_\alpha$ (resp. $B^- \cap G_\alpha$).
D'après \cite[Lemma 3.1.4]{BH}, il existe un sous-tore $T' \subset T$ et un isomorphisme $G_\alpha \cong T' \times \GL_2$ à travers lequel $T \cong T' \times T_\alpha$ avec $T_\alpha$ un tore maximal déployé de $\GL_2$.
On a
\begin{equation*}
s_\alpha s_\beta(\chi)_{|T_\alpha(\Qp)} = s_\alpha(s_\beta(\chi)_{|T_\alpha(\Qp)}) \quad \text{et} \quad s_\alpha s_\beta(\chi)_{|T'(\Qp)} = s_\beta(\chi)_{|T'(\Qp)}.
\end{equation*}

On note $\E_2$ l'unique extension non scindée de $\Indd s_\alpha (s_\beta(\chi)_{|T_\alpha(\Qp)}) \cdot (\varepsilon^{-1} \circ \theta_{|T_\alpha(\Qp)})$ par $\Indd s_\beta(\chi)_{|T_\alpha(\Qp)} \cdot (\varepsilon^{-1} \circ \theta_{|T_\alpha(\Qp)})$ dans la catégorie des représentations continues unitaires admissibles de $\GL_2(\Qp)$ sur $E$ donnée par \cite[Proposition B.2 (i)]{BH}.
En utilisant \cite[Lemma A.6 (ii)]{BH} avec $G_1=\GL_2(\Qp)$, $G_2=T'(\Qp)$, $\Pi_1= \Indd s_\beta(\chi)_{|T_\alpha(\Qp)} \cdot (\varepsilon^{-1} \circ \theta_{|T_\alpha(\Qp)})$, $\Pi_1'= \Indd s_\alpha (s_\beta(\chi)_{|T_\alpha(\Qp)}) \cdot (\varepsilon^{-1} \circ \theta_{|T_\alpha(\Qp)})$ et $\Pi_2=\Pi_2'=s_\beta(\chi)_{|T'(\Qp)} \cdot (\varepsilon^{-1} \circ \theta_{|T'(\Qp)})$, on voit qu'il existe une unique extension non scindée
\begin{equation*}
\E_\alpha \dfn (s_\beta(\chi) \cdot (\varepsilon^{-1} \circ \theta))_{|T'(\Qp)} \otimes_E \E_2
\end{equation*}
de $\Inda s_\alpha s_\beta (\chi) \epsth$ par $\Inda s_\beta(\chi) \epsth$ dans la catégorie des représentations continues unitaires admissibles de $G_\alpha(\Qp)$ sur $E$.
On déduit de \cite[Corollary 4.3.5]{Em1} que $\IndQp[P_\alpha^-] \E_\alpha$ est une extension non scindée de $\IndQp s_\alpha s_\beta (\chi) \epsth$ par $\IndQp s_\beta(\chi) \epsth$ dans la catégorie des représentations continues unitaires admissibles de $G(\Qp)$ sur $E$ et elle est unique en tant que telle d'après \cite[Théorème 1.1 (ii)]{JH}.
Pour prouver le théorème, il suffit donc de montrer que
\begin{equation} \label{extnul}
\Ext_{G(\Qp)}^1 \left(\IndQp[P_\alpha^-] \E_\alpha,\IndQp \chi \epsth\right) = 0.
\end{equation}

Soit $\E_\alpha^0 \subset \E_\alpha$ une boule stable par $G_\alpha(\Qp)$.
Pour tout entier $k\geq1$, on note $\chi_k$ l'image de $\chi$ dans $(\A{k})^\times$.
En utilisant pour tout entier $k \geq 1$ la suite exacte \cite[(3.7.6)]{Em2} pour le triplet $(G,P_\alpha,G_\alpha)$ avec $A=\A{k}$, $U=\E_\alpha^0/\pe^k\E_\alpha^0$ et $V=\IndQp \chi_k \omth$ et l'isomorphisme \eqref{Ord} avec $P=P_\alpha$, $L=G_\alpha$, $A=\A{k}$ et $U=\chi_k \omth$ et en tenant compte de \cite[Lemma 4.1.3]{Em1} et \cite[(B.1) et Proposition B.2]{JH}, on obtient une suite exacte de $E$-espaces vectoriels
\begin{multline} \label{SEExt}
0 \to \Ext^1_{G_\alpha(\Qp)} \left( \E_\alpha,\Inda \chi \epsth \right) \to \\
\Ext^1_{G(\Qp)} \left( \IndQp[P_\alpha^-] \E_\alpha,\IndQp \chi \epsth \right) \to \\
\Hom_{G_\alpha(\Qp)} \Big( \E_\alpha,E \otimes_{\Oe} \varprojlim_k \HOrdQp[P_\alpha] \left( \IndQp \chi_k \omth \right) \Big)
\end{multline}
(pour vérifier que la topologie de la limite projective coïncide avec la topologie $\pe$-adique, on procède comme dans le premier paragraphe de la preuve de \cite[Proposition 3.4.3]{Em1} en remarquant que $\HOrdQp[P_\alpha] \IndQp$ est exact à gauche).

Si $\chi_{|T'(\Qp)} \neq s_\beta(\chi)_{|T'(\Qp)}$, alors en utilisant \cite[Lemma A.6 (i)]{BH} avec $G_1=T'(\Qp)$, $G_2=\GL_2(\Qp)$, $\Pi_1=\chi_{|T'(\Qp)} \cdot (\varepsilon^{-1} \circ \theta_{|T'(\Qp)})$, $\Pi_1'=s_\beta(\chi)_{|T'(\Qp)} \cdot (\varepsilon^{-1} \circ \theta_{|T'(\Qp)})$, $\Pi_2=\Indd \chi_{|T_\alpha(\Qp)} \cdot (\varepsilon^{-1} \circ \theta_{|T_\alpha(\Qp)})$, $\Pi_2'=\E_2$, on obtient
\begin{equation} \label{extanul}
\Ext^1_{G_\alpha(\Qp)} \left( \E_\alpha,\Inda \chi \epsth \right) = 0
\end{equation}
car $\Hom_{G_1}(\Pi_1',\Pi_1)=0$ par hypothèse et $\Ext^1_{G_1}(\Pi_1',\Pi_1)=0$ d'après \cite[Proposition 5.1.6]{JH}.
Sinon, alors $\chi_{|T'(\Qp)} = s_\beta(\chi)_{|T'(\Qp)} = s_\alpha s_\beta(\chi)_{|T'(\Qp)}$, donc $\chi_{|T_\alpha(\Qp)} \neq s_\alpha(s_\beta(\chi)_{|T_\alpha(\Qp)})$ et en utilisant \cite[Lemma A.6 (i)]{BH} avec $G_1=\GL_2(\Qp)$, $G_2=T'(\Qp)$, $\Pi_1=\Indd \chi_{|T_\alpha(\Qp)} \cdot (\varepsilon^{-1} \circ \theta_{|T_\alpha(\Qp)})$, $\Pi_1'=\E_2$, $\Pi_2=\chi_{|T'(\Qp)} \cdot (\varepsilon^{-1} \circ \theta_{|T'(\Qp)})$, $\Pi_2'=s_\beta(\chi)_{|T'(\Qp)} \cdot (\varepsilon^{-1} \circ \theta_{|T'(\Qp)})$, on obtient encore l'égalité \eqref{extanul} car d'une part $\Hom_{G_1}(\Pi_1',\Pi_1)=0$ et d'autre part ou bien $\chi_{|T_\alpha(\Qp)} \neq s_\beta(\chi)_{|T_\alpha(\Qp)}$ d'où $\Ext_{G_1}^1(\Pi_1',\Pi_1)=0$ d'après \cite[Proposition 4.3.15 (1)]{Em2}, ou bien $\chi_{|T'(\Qp)} \neq s_\beta(\chi)_{|T'(\Qp)}$ d'où $\Hom_{G_2}(\Pi_2',\Pi_2)=0$.

En utilisant le théorème \ref{theo:HnOrd} avec $P=P_\alpha$, $L=G_\alpha$, $n=1$, $A=\A{k}$ et $U = \chi_k \omth$ pour tout entier $k \geq 1$, on obtient un isomorphisme continu $G_\alpha(\Qp)$-équivariant
\begin{multline} \label{H1OrdaInd}
E \otimes_{\Oe} \varprojlim_k \HOrdQp[P_\alpha] \left( \IndQp \chi_k \omth \right) \\
\cong \bigoplus_{\gamma \in \Delta-\{\alpha\}} E \otimes_{\Oe} \varprojlim_k \HOrdQp[P_\alpha] \left( \cindQp^{(B^- \dot{s}_\gamma P_\alpha)(\Qp)} \chi_k \omth \right)
\end{multline}
et pour tout $\gamma \in \Delta-\{\alpha\}$, les morphismes \eqref{HnOrd} avec $P=P_\alpha$, $L=L_\alpha$, $n=1$, $A=\A{k}$ et $U = \chi_k \omth$ pour tout entier $k \geq 1$ donnent, en tenant compte du fait que $s_\gamma(\theta)+\gamma=\theta$ (voir la preuve de \cite[Corollaire 4.2.7]{JH}), un morphisme continu $G_\alpha(\Qp)$-équivariant
\begin{multline} \label{H1Orda}
\Inda s_\gamma(\chi) \epsth \\
\to E \otimes_{\Oe} \varprojlim_k \HOrdQp[P_\alpha] \left( \cindQp^{(B^- \dot{s}_\gamma P_\alpha)(\Qp)} \chi_k \omth \right).
\end{multline}

Soit $\gamma \in \Delta-\{\alpha\}$.
Le morphisme \eqref{H1Orda} est injectif en restriction à la sous-$B_\alpha(\Qp)$-représentation fermée
\begin{equation*}
E \otimes_{\Oe} \varprojlim_k \Fil^0_{B_\alpha} \Inda s_\gamma(\chi_k) \omth.
\end{equation*}
Cette dernière est topologiquement irréductible (car résiduellement irréductible, voir la preuve de la proposition \ref{prop:HnOrdirr}), de codimension $1$ et son image par le morphisme \eqref{H1Orda} est encore de codimension $1$.
De plus $T(\Qp)$ agit sur ces représentations unidimensionnelles à travers le même caractère $s_\alpha(s_\gamma(\chi) \epsth)$.
Comme un caractère de $G_\alpha(\Qp)$ est déterminé par sa restriction à $T(\Qp)$, on en déduit un isomorphisme continu $G_\alpha(\Qp)$-équivariant
\begin{multline} \label{H1Ordass}
\Big( E \otimes_{\Oe} \varprojlim_k \HOrdQp[P_\alpha] \left( \cindQp^{(B^- \dot{s}_\gamma P_\alpha)(\Qp)} \chi_k \omth \right) \Big)^\sms \\
\cong \left( \Inda s_\gamma(\chi) \epsth \right)^\sms
\end{multline}
où l'exposant $^\sms$ désigne la semi-simplifiée dans la catégorie des représentations continues unitaires admissibles de $G_\alpha(\Qp)$ sur $E$.

Pour tout $\gamma \in \Delta-\{\alpha\}$, on a $s_\alpha s_\beta(\chi) \neq s_\gamma(\chi)$ (par hypothèse si $\gamma=\beta$ et d'après le point (i) du lemme \ref{lemm:gen} sinon) donc les représentations $\Inda s_\alpha s_\beta (\chi) \epsth$ et $\Inda s_\gamma (\chi) \epsth$ n'ont aucun constituant irréductible en commun.
Comme l'extension $\E_\alpha$ n'est pas scindée et la représentation $\Inda s_\beta(\chi) \epsth$ est topologiquement irréductible, on déduit de l'isomorphisme \eqref{H1OrdaInd} et des isomorphismes \eqref{H1Ordass} pour tout $\gamma \in \Delta-\{\alpha\}$ que
\begin{equation*}
\Hom_{G_\alpha(\Qp)} \Big( \E_\alpha,E \otimes_{\Oe} \varprojlim_k \HOrdQp[P_\alpha] \left( \IndQp \chi_k \omth \right) \Big)=0.
\end{equation*}
En utilisant la suite exacte \eqref{SEExt} et en tenant compte de l'égalité \eqref{extanul}, on en déduit l'égalité \eqref{extnul}.

\medskip

\emph{2\up{d} cas : $\chi'=\chi$ ou $\chi'=\chi''$.}

Dans ce cas, on déduit de \cite[Théorème 1.1 (i)]{JH} que $\chi'' = s_\alpha(\chi)$ avec $\alpha \in \Delta$ telle que $\chi \circ \alpha^\vee \neq 1$ (voir le point (i) de la remarque \ref{rema:gen}) et $\chi'$ est faiblement générique par hypothèse.
On prouve le théorème lorsque $\chi'=\chi$ (la preuve est similaire lorsque $\chi'=\chi''$).

Soient $\Et$ une auto-extension non scindée de $\IndQp \chi \epsth$ et $\Ct$ une extension non scindée de $\IndQp s_\alpha(\chi) \epsth$ par $\Et$ dans la catégorie des représentations continues unitaires admissibles de $G(\Qp)$ sur $E$.
Pour prouver le théorème, il suffit de montrer que $\Ct$ n'est pas une chaîne comme dans l'énoncé.

D'après \cite[Théorème 1.1 (iii)]{JH}, il existe une auto-extension non scindée $\E$ de $\chi \epsth$ dans la catégorie des représentations continues unitaires admissibles de $T(\Qp)$ sur $E$ et un isomorphisme continu $G(\Qp)$-équivariant $\Et \cong \IndQp \E$.

Soient $\E^0 \subset \E$ une boule stable par $T(\Qp)$.
Pour tout entier $k\geq1$, on note $\chi_k$ l'image de $\chi$ dans $(\A{k})^\times$.
En utilisant pour tout entier $k \geq 1$ la suite exacte \cite[(3.7.6)]{Em2} pour le triplet $(G,P_\alpha,G_\alpha)$ avec $A=\A{k}$, $U=\chi_k \omth$ et $V=\IndQp \E^0/\pe^k\E^0$, l'isomorphisme \eqref{Ord} avec $P=P_\alpha$, $L=G_\alpha$, $A=\A{k}$ et $U=\E^0/\pe^k\E^0$ et \cite[Corollaire 4.2.4 (i)]{JH} avec $A=\A{k}$ et $U=\E^0/\pe^k\E^0$ et en tenant compte de \cite[Lemma 4.1.3]{Em1} et \cite[(B.1) et Proposition B.2]{JH}, on obtient une suite exacte de $E$-espaces vectoriels
\begin{multline*}
0 \to \Ext^1_{T(\Qp)} \left( s_\alpha(\chi) \epsth,\E \right) \\
\to \Ext^1_{G(\Qp)} \left( \IndQp s_\alpha(\chi) \epsth,\Et \right) \\
\to \bigoplus_{\beta \in \Delta} \Hom_{T(\Qp)} \left( s_\alpha(\chi) \epsth,\E^\beta \otimes (\varepsilon^{-1} \circ \beta) \right).
\end{multline*}
Comme $s_\alpha(\chi) \neq \chi$, le premier terme non trivial de la suite exacte est nul d'après \cite[Proposition 5.1.6]{JH}.
Pour tout $\beta \in \Delta$, la représentation $\E^\beta \otimes (\omega^{-1} \circ \beta)$ est une auto-extension non scindée de $s_\beta(\chi) \epsth$ car $s_\beta(\theta)+\beta=\theta$ (voir la preuve de \cite[Corollaire 4.2.7]{JH}), donc la somme directe est de dimension $1$ (seul le terme correspondant à $\alpha$ est non nul car $\chi$ est faiblement générique).
On en déduit que le second terme non trivial de la suite exacte est de dimension $1$.
Comme $s_\alpha(\chi) \neq \chi$, on a une injection $E$-linéaire
\begin{multline*}
\Ext^1_{G(\Qp)} \left( \IndQp s_\alpha(\chi) \epsth,\IndQp \chi \epsth \right) \\
\hookrightarrow \Ext^1_{G(\Qp)} \left( \IndQp s_\alpha(\chi) \epsth,\Et \right)
\end{multline*}
dont la source est de dimension $1$ d'après \cite[Théorème 1.1 (ii)]{JH}, donc c'est un isomorphisme.
On en conclut qu'il existe une extension non scindée $\Et_\alpha$ de $\IndQp s_\alpha(\chi) \epsth$ par $\IndQp \chi \epsth$ dans la catégorie des représentations continues unitaires admissibles de $G(\Qp)$ sur $E$ et une injection continue $G(\Qp)$-équivariante $\Et_\alpha \hookrightarrow \Ct$ dont le quotient est isomorphe à $\IndQp \chi \epsth$, donc $\Ct$ n'est pas une chaîne comme dans l'énoncé.
\end{proof}

\begin{rema}
L'égalité \eqref{extnul} est encore vraie sans supposer les séries principales de l'énoncé topologiquement irréductibles, mais en supposant seulement $\chi \circ \beta^\vee \neq 1$ et $\chi \circ s_\beta(\alpha)^\vee \not \in \{1,\varepsilon\}$ (auquel cas il existe encore une unique extension non scindée $\E_2$ par une preuve identique à celle de \cite[Proposition B.2 (i)]{BH}). L'hypothèse $\chi \circ s_\beta(\alpha)^\vee \neq \varepsilon$ ne serait pas nécessaire non plus si la conjecture \ref{conj:HnOrd} était démontrée pour $P=P_\alpha$ et $n=1$.
\end{rema}

\subsection{La construction de Breuil-Herzig}

On garde les hypothèses précédentes sur $F$ et $G$ ainsi que la notation $\theta$.
On établit certaines propriétés de la construction de Breuil-Herzig dans la catégorie des représentations continues unitaires admissibles de $G(\Qp)$ sur $E$.
Les résultats analogues modulo $p$ se démontrent de façon analogue.

Si $\Pi$ est une représentation continue unitaire admissible de $G(\Qp)$ sur $E$, on note $\soc^\bullet \Pi$ sa filtration par le socle (définie par $\soc^{-1} \Pi = 0$ et $\soc^k \Pi/\soc^{k-1} \Pi \cong \soc(\Pi/\soc^{k-1} \Pi)$ pour tout $k \in \N$) et $\rad \Pi$ son radical, c'est-à-dire le noyau de la projection sur le cosocle (voir \cite[§~I.1]{Alp}).
On note que $\soc \Pi$ est toujours de longueur finie (car $\Pi$ est admissible) et que $\soc \Pi = 0$ si et seulement si $\Pi = 0$ (voir \cite[Lemma 5.8]{PasAdm}).
En revanche, on peut avoir $\rad \Pi = \Pi$ et $\Pi \neq 0$.

\medskip

Soient $\chi : T(\Qp) \to \Oe^\times \subset E^\times$ un caractère continu unitaire et $\Psi \subset \Phi^+$ un sous-ensemble fermé\footnote{Un sous-ensemble $\Psi \subset \Phi^+$ est fermé si $\alpha+\beta \in \Phi^+ \Rightarrow \alpha+\beta \in \Psi$  pour tous $\alpha,\beta \in \Psi$.}.
On pose
\begin{equation*}
W_\Psi \dfn \left\{ w \in W \mid w(\Psi) \subset \Phi^+\right\}.
\end{equation*}
Pour tout $w_\Psi \in W_\Psi$ et pour tout $I \subset \Delta \cap w_\Psi(\Psi)$ constitué de racines deux à deux orthogonales, on pose
\begin{equation*}
C_{w_\Psi,I} \dfn \IndQp \big( \prod_{\alpha \in I} s_\alpha \big) w_\Psi(\chi) \epsth.
\end{equation*}
D'après \cite[Conjecture 3.1.2]{BH}, ces représentations devraient être topologiquement irréductibles lorsque $\chi \circ \alpha^\vee \neq \varepsilon^{\pm1}$ pour tout $\alpha \in \Phi^+$.
La conjecture analogue modulo $p$ est vraie d'après \cite[Théorème 4]{Oll} lorsque $G=\GL_n$ et \cite[Theorem 1.3]{Abe} dans le cas général déployé.
En particulier si la réduction $\overline{\chi} : T(\Qp) \to \ke^\times$ de $\chi$ modulo $\pe$ vérifie $\overline{\chi} \circ \alpha^\vee \neq \omega^{\pm1}$ pour tout $\alpha \in \Phi^+$, alors ces représentations sont topologiquement irréductibles.

\medskip

Soit $w_\Psi \in W_\Psi$.
On rappelle la construction de Breuil-Herzig (voir \cite[§~3.3]{BH} avec $C_\rho$ et $w_{C_\rho}^{-1}$ au lieu de $\Psi$ et $w_\Psi$).
On suppose $\chi$ générique et $C_{w_\Psi,I}$ topologiquement irréductible pour tout $I \subset \Delta \cap w_\Psi(\Psi)$ constitué de racines deux à deux orthogonales.
Par généricité de $w_\Psi(\chi)$, ces représentations sont deux à deux non isomorphes d'après le point (ii) du lemme \ref{lemm:gen}.

\begin{rema} \label{rema:ext}
Soient $I,I' \subset \Delta \cap w_\Psi(\Psi)$ constitués de racines deux à deux orthogonales.
En utilisant le point (ii) du lemme \ref{lemm:gen} avec $(\prod_{\alpha \in I \cap I'} s_\alpha)w_\Psi(\chi)$ au lieu de $\chi$, on déduit de \cite[Théorème 1.1 (i)]{JH} qu'il existe une extension non scindée de $C_{I',w_\Psi}$ par $C_{I,w_\Psi}$ si et seulement si $\card(I \cup I' - I \cap I')=1$ ou $I'=I$.
\end{rema}

Soit $I \subset \Delta \cap w_\Psi(\Psi)$ constitué de racines deux à deux orthogonales.
On note $G_I \subset G$ le sous-groupe fermé engendré par $T$ et les sous-groupes radiciels correspondant aux racines dans $\pm I$.
D'après \cite[Lemma 3.1.4]{BH}, il existe un sous-tore $T' \subset T$ et un isomorphisme $G_I \cong T' \times \GL_2^I$ à travers lequel $T \cong T' \times \prod_{\alpha \in I} T_\alpha$ avec $T_\alpha$ un tore maximal déployé dans la copie de $\GL_2$ correspondant à $\alpha$.
Pour tout $\alpha \in I$, on note $\E_\alpha$ l'unique extension non scindée de $\Indd s_\alpha (w_\Psi(\chi)_{|T_\alpha(\Qp)}) \cdot (\varepsilon^{-1} \circ \theta_{|T_\alpha(\Qp)})$ par $\Indd w_\Psi(\chi)_{|T_\alpha(\Qp)} \cdot (\varepsilon^{-1} \circ \theta_{|T_\alpha(\Qp)})$ donnée par \cite[Proposition B.2 (i)]{BH} et on pose
\begin{equation*}
\Pi(\chi)_{w_\Psi,I} \dfn \Ind_{(B^-G_I)(\Qp)}^{G(\Qp)} \left( (w_\Psi(\chi) \epsth)_{|T'(\Qp)} \otimes_E \widehat{\bigotimes}_{E,\alpha \in I} \E_\alpha \right).
\end{equation*}
En utilisant \cite[Lemma A.6 (ii)]{BH}, on voit que les constituants irréductibles de $\Pi(\chi)_{w_\Psi,I}$ sont exactement les représentations $(C_{w_\Psi,I'})_{I' \subset I}$ avec multiplicité un et que pour tout $I' \subset I$, le degré de $C_{w_\Psi,I'}$ dans le gradué de la filtration par le socle de $\Pi(\chi)_{w_\Psi,I}$ est $\card I'$.
Le treillis des sous-représentations fermées de $\Pi(\chi)_{w_\Psi,I}$ est le treillis des parties fermées inférieurement\footnote{Une partie fermée inférieurement d'un ensemble ordonné $(X,\leq)$ est un sous-ensemble $Y \subset X$ vérifiant $x \leq y \Rightarrow x \in Y$ pour tous $x \in X$ et $y \in Y$.} de $(I,\subset)$.
En particulier, il a une structure d'\og hypercube \fg.

Pour tout $I \subset \Delta \cap w_\Psi(\Psi)$ constitué de racines deux à deux orthogonales et pour tout $I' \subset I$, on a une injection continue $G(\Qp)$-équivariante $\Pi(\chi)_{w_\Psi,I'} \hookrightarrow \Pi(\chi)_{w_\Psi,I}$ unique à multiplication par un scalaire près.
On fixe un système d'injections compatibles et on pose
\begin{equation*}
\Pi(\chi)_{\Psi,w_\Psi} \dfn \varinjlim_I \Pi(\chi)_{w_\Psi,I}
\end{equation*}
avec $I \subset \Delta \cap w_\Psi(\Psi)$ parmi les sous-ensembles de racines deux à deux orthogonales.
La représentation $\Pi(\chi)_{\Psi,w_\Psi}$ est de longueur finie sans multiplicité et elle est engendrée par les images des injections continues $G(\Qp)$-équivariantes $\Pi(\chi)_{w_\Psi,I} \hookrightarrow \Pi(\chi)_{\Psi,w_\Psi}$ avec $I \subset \Delta \cap w_\Psi(\Psi)$ parmi les sous-ensembles de racines deux à deux orthogonales.

\begin{theo} \label{theo:ord}
Soient $\chi : T(\Qp) \to \Oe^\times \subset E^\times$ un caractère continu unitaire générique, $\Psi \subset \Phi^+$ un sous-ensemble fermé et $w_\Psi \in W_\Psi$.
On suppose que pour tout $I \subset \Delta \cap w_\Psi(\Psi)$ constitué de racines deux à deux orthogonales, la représentation $C_{w_\Psi,I}$ est topologiquement irréductible.
Alors $\Pi(\chi)_{\Psi,w_\Psi}$ est la plus grande représentation continue unitaire admissible de $G(\Qp)$ sur $E$ dont le socle est $C_{w_\Psi,\emptyset}$ et dont les autres sous-quotients irréductibles sont des séries principales distinctes de $C_{w'_\Psi,\emptyset}$ pour tout $w'_\Psi \in W_\Psi$.
\end{theo}

\begin{rema}
Le théorème avec $\Psi=\Phi^+$ donne une classification de toutes les représentations continues unitaires admissibles de $G(\Qp)$ sur $E$ dont le socle est $\IndQp \chi \epsth$ et dont les autres sous-quotients irréductibles sont des séries principales distinctes du socle : ce sont exactement les sous-représentations fermées de $\Pi(\chi)_{\Phi^+,1}$.
\end{rema}

\begin{proof}
Soit $\Pi$ une représentation continue unitaire admissible de $G(\Qp)$ sur $E$ satisfaisant les conditions de l'énoncé.
Il faut montrer que l'on a une injection $\Pi \hookrightarrow \Pi(\chi)_{\Psi,w_\Psi}$.

Comme $\Pi(\chi)_{\Psi,w_\Psi}$ est de longueur finie, on peut procéder par induction.
Soient $\Pi' \subset \Pi$ une sous-représentation et $\chi' : T(\Qp) \to \Oe^\times \subset E^\times$ un caractère continu unitaire vérifiant $\chi' \neq w'_\Psi(\chi)$ pour tout $w'_\Psi \in W_\Psi$.
On suppose que l'on a une injection $\Pi' \hookrightarrow \Pi(\chi)_{\Psi,w_\Psi}$ et une suite exacte courte non scindée
\begin{equation*}
0 \to \Pi' \to \Pi \to C' \to 0
\end{equation*}
avec $C' \cong \IndQp \chi' \epsth$ topologiquement irréductible.

\medskip

\emph{Étape 1 : on définit une sous-représentation $\Pi_{C'}$ de $\Pi$ et un constituant irréductible $C$ de $\Pi_{C'}$.}

Soit $\Pi_{C'} \subset \Pi$ l'unique sous-représentation dont le cosocle est $C'$ à travers la composée $\Pi_{C'} \hookrightarrow \Pi \twoheadrightarrow C'$.
Par hypothèse, $C'$ n'est pas dans le socle de $\Pi$ donc $\rad \Pi_{C'} \neq 0$ et on a une injection $\rad \Pi_{C'} \subset \Pi' \hookrightarrow \Pi(\chi)_{\Psi,w_\Psi}$.

Soit $C$ un facteur irréductible dans le cosocle de $\rad \Pi_{C'}$.
Il existe $I \subset \Delta \cap w_\Psi(\Psi)$ constitué de racines deux à deux orthogonales tel que $C \cong C_{w_\Psi,I}$.
De plus, l'unique quotient de $\Pi_{C'}$ ayant pour constituants irréductibles $C$ et et $C'$ est une extension non scindée de $C'$ par $C$ (elle est induite par l'extension non scindée $\Pi_{C'}$ de $C'$ par $\rad \Pi_{C'}$ et la surjection $\rad \Pi_{C'} \twoheadrightarrow C$).

\medskip

\emph{Étape 2 : on montre qu'il existe $I' \subset \Delta \cap w_\Psi(\Psi)$ constitué de racines deux à deux orthogonales tel que $C' \cong C_{w_\Psi,I'}$.}

On suppose d'abord $I=\emptyset$, donc $C \cong C_{w_\Psi,\emptyset}$.
Par hypothèse, $C' \not \cong C_{w_\Psi,\emptyset}$ donc d'après \cite[Théorème 1.1 (i)]{JH} il existe $\alpha \in \Delta$ tel que $\chi' = s_\alpha w_\Psi(\chi)$.
Toujours par hypothèse, $s_\alpha w_\Psi \not \in W_\Psi$ donc $\alpha \in \Delta \cap w_\Psi(\Psi)$.
Ainsi, $C' \cong C_{w_\Psi,I'}$ avec $I'=\{\alpha\}$.

On suppose maintenant $I \neq \emptyset$ et on fixe $\beta \in I$.
Si $I=\{\beta\}$, alors $C_{w_\Psi,I-\{\beta\}} \cong C_{w_\Psi,\emptyset} \not \cong C'$ par hypothèse.
Sinon quitte à changer $\beta \in I$, on peut supposer $C' \not \cong C_{w_\Psi,I-\{\beta\}}$.
L'unique quotient de $\Pi_{C'}$ ayant pour socle $C_{w_\Psi,I-\{\beta\}}$ a pour cosocle $C'$ et admet $C$ comme constituant irréductible intermédiaire.
On déduit théorème \ref{theo:chaine} qu'il admet un autre constituant irréductible intermédiaire $C'' \not \cong C$ parmi les facteurs irréductibles du cosocle de $\rad \Pi_{C'}$.
On a donc $C'' \cong C_{w_\Psi,I''}$ avec $I'' \subset \Delta \cap w_\Psi(\Psi)$ constitué de racines deux à deux orthogonales et il existe $\alpha \in I''$ tel que $I''-\{\alpha\}=I-\{\beta\}$ et $\alpha \neq \beta$.
En utilisant la remarque \ref{rema:ext}, on voit qu'il n'existe pas d'extension non scindée de $C$ par $C''$ ou de $C''$ par $C$.
Comme $C'$ admet des extensions non scindées par $C$ et $C''$, on a $C' \not \cong C$ et $C' \not \cong C''$ et on déduit de \cite[Théorème 1.1 (i)]{JH} qu'il existe $\alpha',\beta' \in \Delta$ tels que
\begin{equation*}
\chi' = s_{\alpha'} s_\alpha \big(\prod_{\gamma \in I-\{\beta\}} s_\gamma\big) w_\Psi(\chi) = s_{\beta'} s_\beta \big(\prod_{\gamma \in I-\{\beta\}} s_\gamma\big) w_\Psi(\chi)
\end{equation*}
avec $\alpha \neq \alpha'$ et $\beta \neq \beta'$ (car $C' \not \cong C_{w_\Psi,I-\{\beta\}}$ par hypothèse).
En utilisant le point (iii) du lemme \ref{lemm:gen} avec $(\prod_{\gamma \in I-\{\beta\}} s_\gamma) w_\Psi(\chi)$ au lieu de $\chi$, on en déduit que $s_{\alpha'} s_\alpha = s_{\beta'} s_\beta$, donc $\alpha'=\beta$, $\beta'=\alpha$ et $\alpha \perp \beta$.
Ainsi, $C' \cong C_{w_\Psi,I'}$ avec $I' = I \cup \{\alpha\}$.

\medskip

\emph{Étape 3 : on montre que $\rad \Pi_{C'} \cong \rad \Pi(\chi)_{w_\Psi,I'}$.}

On déduit de l'étape 2 que le cosocle de $\rad \Pi_{C'}$ est un facteur direct du cosocle de $\rad \Pi(\chi)_{w_\Psi,I'}$, d'où une injection $\rad \Pi_{C'} \hookrightarrow \rad \Pi(\chi)_{w_\Psi,I'}$.
Pour montrer que c'est un isomorphisme, il suffit de vérifier que pour tout $\beta \in I'$, la représentation $C_{w_\Psi,I'-\{\beta\}}$ apparaît dans $\Pi_{C'}$.

Soit $\beta \in I'$.
Par l'absurde, on suppose que $C_{w_\Psi,I'-\{\beta\}}$ n'apparaît pas dans $\Pi_{C'}$.
D'après l'étape 2, on a $I' = I \cup \{\alpha\}$ et nécessairement $\alpha \neq \beta$.
En utilisant la remarque \ref{rema:ext}, on voit que $C$ et $C_{w_\Psi,I'-\{\beta\}}$ sont les seuls constituants irréductibles $C''$ de $\Pi(\chi)_{\Psi,w_\Psi}$ tels qu'il existe des extensions non scindées de $C'$ par $C''$ et de $C''$ par $C_{w_\Psi,I-\{\beta\}}$.
Ainsi, l'unique quotient de $\Pi_{C'}$ ayant pour socle $C_{w_\Psi,I-\{\beta\}}$ a pour cosocle $C'$ et pour unique constituant irréductible intermédiaire $C$.
Or, une telle représentation n'existe pas d'après le théorème \ref{theo:chaine}, d'où la contradiction.

\medskip

\emph{Étape 4 : on montre que $\Pi_{C'} \cong \Pi(\chi)_{w_\Psi,I'}$.}

En utilisant les étapes 2 et 3, on voit qu'il suffit de montrer que
\begin{equation} \label{Ext1I}
\dim_E \Ext^1_{G(\Qp)} \left( C_{w_\Psi,I'},\rad \Pi(\chi)_{w_\Psi,I'} \right) = 1.
\end{equation}
On procède par récurrence sur $\card I'$.
Si $I'=\{\alpha\}$, alors $\rad \Pi(\chi)_{w_\Psi,I'} \cong C_{w_\Psi,\emptyset}$ et l'égalité \eqref{Ext1I} résulte de \cite[Théorème 1.1 (ii)]{JH}.
On suppose $\card I'>1$ et on fixe $\alpha \in I'$.
Dans ce cas, la suite exacte courte non scindée
\begin{equation*}
0 \to \Pi(\chi)_{w_\Psi,I'-\{\alpha\}} \to \Pi(\chi)_{w_\Psi,I'} \to \Pi(s_{w_\Psi^{-1}(\alpha)}(\chi))_{w_\Psi,I'-\{\alpha\}} \to 0
\end{equation*}
induit une suite exacte courte non scindée
\begin{equation*}
0 \to \Pi(\chi)_{w_\Psi,I'-\{\alpha\}} \to \rad \Pi(\chi)_{w_\Psi,I'} \to \rad \Pi(s_{w_\Psi^{-1}(\alpha)}(\chi))_{w_\Psi,I'-\{\alpha\}} \to 0,
\end{equation*}
d'où une suite exacte
\begin{multline*}
\Ext_{G(\Qp)}^1 \left( C_{w_\Psi,I'},\Pi(\chi)_{w_\Psi,I'-\{\alpha\}} \right) \to \Ext_{G(\Qp)}^1 \left( C_{w_\Psi,I'},\rad \Pi(\chi)_{w_\Psi,I'} \right) \\
\to \Ext_{G(\Qp)}^1 \left( C_{w_\Psi,I'},\rad \Pi(s_{w_\Psi^{-1}(\alpha)}(\chi))_{w_\Psi,I'-\{\alpha\}} \right).
\end{multline*}
On montre que le premier terme de la suite exacte est nul.
Par l'absurde, on suppose qu'il existe une extension non scindée $\widetilde{\Pi}$ de $C_{w_\Psi,I'}$ par $\Pi(\chi)_{w_\Psi,I'-\{\alpha\}}$.
En utilisant la remarque \ref{rema:ext}, on voit que pour tout $\beta \in I'-\{\alpha\}$, l'unique quotient de $\widetilde{\Pi}$ ayant pour socle $C_{w_\Psi,I'-\{\alpha,\beta\}}$ a pour cosocle $C_{w_\Psi,I'}$ et pour unique constituant irréductible intermédiaire $C_{w_\Psi,I'-\{\alpha\}}$.
Or, une telle représentation n'existe pas d'après le théorème \ref{theo:chaine}, d'où la contradiction.
Comme le dernier terme de la suite exacte est de dimension $1$ par hypothèse de récurrence avec $s_{w_\Psi^{-1}(\alpha)}(\chi)$ au lieu de $\chi$, on en déduit l'égalité \eqref{Ext1I}.

\medskip

\emph{Étape 5 : on montre que $\Pi$ est sans multiplicité.}

Comme $\Pi'$ est sans multiplicité, il suffit de vérifier que $C'$ apparaît avec multiplicité un dans $\Pi$.
Si $C'$ apparaît dans $\Pi'$, alors son degré dans le gradué de la filtration par le socle de $\Pi'$ est $d = \card I'$ (car on a une injection $\Pi' \hookrightarrow \Pi(\chi)_{\Psi,w_\Psi}$), donc il suffit de vérifier que $C'$ apparaît avec multiplicité un dans $\soc^d \Pi/\soc^{d-1} \Pi$.
Comme $\soc^{d-1} \Pi \cong \soc^{d-1} \Pi'$, on a une suite exacte
\begin{multline*}
\Hom_{G(\Qp)} \left( C',\soc^d \Pi \right) \to \Hom_{G(\Qp)} \left( C',\soc^d \Pi/\soc^{d-1} \Pi \right) \\
\to \Ext^1_{G(\Qp)} \left( C',\soc^{d-1} \Pi' \right).
\end{multline*}
Le premier terme est nul (car $C'$ n'est pas dans le socle de $\Pi$ par hypothèse), donc il suffit de montrer que le dernier terme de cette suite exacte est de dimension au plus $1$.
Or, on a une suite exacte
\begin{multline*}
\Ext^1_{G(\Qp)} \left( C',\rad \Pi_{C'} \right) \to \Ext^1_{G(\Qp)} \left( C',\soc^{d-1} \Pi' \right) \\
\to \Ext^1_{G(\Qp)} \left( C',\soc^{d-1} \Pi'/\rad \Pi_{C'} \right).
\end{multline*}
En utilisant la remarque \ref{rema:ext}, on voit que le dernier terme de la suite exacte est nul.
En utilisant l'égalité \eqref{Ext1I} avec les étapes 2 et 3, on en déduit que le second terme de la suite exacte est de dimension au plus $1$.

\medskip

\emph{Étape finale : on termine la démonstration.}

On a des injections $\Pi' \hookrightarrow \Pi(\chi)_{\Psi,w_\Psi}$ et $\Pi_{C'} \cong \Pi(\chi)_{w_\Psi,I'} \hookrightarrow \Pi(\chi)_{\Psi,w_\Psi}$, que l'on peut choisir égales en restriction à $\Pi' \cap \Pi_{C'} \cong \rad \Pi(\chi)_{w_\Psi,I'}$ (car l'injection $\rad \Pi(\chi)_{w_\Psi,I'} \hookrightarrow \Pi(\chi)_{\Psi,w_\Psi}$ est unique à multiplication par un scalaire près).
Comme $\Pi=\Pi'+\Pi_{C'}$ est de longueur finie sans multiplicité, les injections précédentes induisent une injection $\Pi \hookrightarrow \Pi(\chi)_{\Psi,w_\Psi}$.
\end{proof}

\begin{coro} \label{coro:conjBH}
\cite[Conjecture 3.5.1]{BH} est vraie.
\end{coro}

\begin{proof}
La conjecture de Breuil-Herzig affirme que la représentation $\Pi(\chi)_{\Psi,w_\Psi}$ est unique étant donné le gradué de sa filtration par le socle.
Le théorème \ref{theo:ord} répond à une question plus forte de Breuil-Herzig : la représentation $\Pi(\chi)_{\Psi,w_\Psi}$ est l'unique représentation continue unitaire admissible de $G(\Qp)$ sur $E$ de longueur finie sans multiplicité, dont le socle est $C_{w_\Psi,\emptyset}$ et dont les constituants irréductibles sont exactement les $C_{w_\Psi,I}$ avec $I \subset \Delta \cap w_\Psi(\Psi)$ parmi les sous-ensembles de racines deux à deux orthogonales.
\end{proof}

\begin{rema}
Pour la construction de $\Pi(\chi)_{\Psi,w_\Psi}$ et la preuve de la conjecture de Breuil-Herzig (ou de la question plus forte), il suffit de supposer $\chi \circ \alpha^\vee \neq 1$ pour tout $\alpha \in w_\Psi^{-1}(\Delta) \cap \Psi$ et $C_{w_\Psi,I}$ topologiquement irréductible pour tout $I \subset \Delta \cap w_\Psi(\Psi)$ constitué de racines deux à deux orthogonales.
\end{rema}

On suppose $\chi$ fortement générique et $C_{w_\Psi,I}$ topologiquement irréductible pour tout $w_\Psi \in W_\Psi$ et pour tout $I \subset \Delta \cap w_\Psi(\Psi)$ constitué de racines deux à deux orthogonales.
Par forte généricité de $\chi$, ces représentations sont deux à deux non isomorphes (car les éléments $w \in W$ de la forme $w=(\prod_{\alpha \in I} s_\alpha)w_\Psi$ avec $w_\Psi \in W_\Psi$ et $I \subset \Delta \cap w_\Psi(\Psi)$ constitué de racines deux à deux orthogonales sont deux à deux distincts puisque $I = \Delta \cap w(-\Psi)$).
On pose
\begin{equation*}
\Pi(\chi)_\Psi \dfn \bigoplus_{w_\Psi \in W_\Psi} \Pi(\chi)_{\Psi,w_\Psi}.
\end{equation*}

\begin{coro} \label{coro:ord}
Soient $\chi : T(\Qp) \to \Oe^\times \subset E^\times$ un caractère continu unitaire fortement générique et $\Psi \subset \Phi^+$ un sous-ensemble fermé.
On suppose que pour tout $w_\Psi \in W_\Psi$ et pour tout $I \subset \Delta \cap w_\Psi(\Psi)$ constitué de racines deux à deux orthogonales, la représentation $C_{w_\Psi,I}$ est topologiquement irréductible.
Alors $\Pi(\chi)_\Psi$ est la plus grande représentation continue unitaire admissible de $G(\Qp)$ sur $E$ satisfaisant
\begin{enumerate}
\item $\soc \Pi(\chi)_\Psi \cong \bigoplus_{w_\Psi \in W_\Psi} \IndQp w_\Psi(\chi) \epsth$ ;
\item les sous-quotients irréductibles de $\Pi(\chi)_\Psi$ sont des séries principales ;
\item les facteurs irréductibles de $\soc \Pi(\chi)_\Psi$ apparaissent avec multiplicité un dans $\Pi(\chi)_\Psi$ (donc seulement dans le socle).
\end{enumerate}
\end{coro}

\begin{proof}
Soit $\Pi$ une représentation continue unitaire admissible de $G(\Qp)$ sur $E$ satisfaisant les points (i), (ii) et (iii) de l'énoncé.
Pour tout $w_\Psi \in W_\Psi$, on note $\Pi_{w_\Psi}$ l'unique quotient de $\Pi$ dont le socle est $C_{w_\Psi,\emptyset}$.
D'après le théorème \ref{theo:ord}, on a une injection $\Pi_{w_\Psi} \hookrightarrow \Pi(\chi)_{\Psi,w_\Psi}$.
En particulier, $\bigoplus_{w_\Psi \in W_\Psi} \Pi_{w_\Psi}$ est de longueur finie sans multiplicité, d'où un isomorphisme $\Pi \cong \bigoplus_{w_\Psi \in W_\Psi} \Pi_{w_\Psi}$.
Les injections précédentes induisent donc une injection $\Pi \hookrightarrow \Pi(\chi)_\Psi$.
\end{proof}

\subsection{\texorpdfstring{Conjecture sur les extensions modulo $p$}{Conjecture sur les extensions modulo p}}

On garde seulement les hypothèses précédentes sur $G$.
Soient $P \subset G$ un sous-groupe parabolique standard et $L \subset P$ le sous-groupe de Levi standard.
On note $P^- \subset G$ le sous-groupe parabolique opposé à $P$ par rapport à $L$, on reprend les notations de la sous-section \ref{ssec:fil} et on pose
\begin{equation*}
\Delta_L^\perp \dfn \{\alpha \in \Delta \mid \text{$\alpha \perp \beta$ pour tout $\beta \in \Delta_L$}\}.
\end{equation*}
Tout $\alpha \in \Delta_L^\perp$ se prolonge de façon unique en un caractère algébrique de $L$ et l'action par conjugaison de $\ds_\alpha$ stabilise $L$.
Si $\pi$ est une représentation lisse de $L(F)$ sur $A$ et $\alpha \in \Delta_L^\perp$, on note $\pi^\alpha$ la représentation lisse de $L(F)$ sur $A$ dont le $A$-module sous-jacent est $\pi$ et sur lequel $l \in L(F)$ agit à travers $\ds_\alpha l \ds_\alpha$ ($\pi^\alpha$ ne dépend pas du choix du représentant $\ds_\alpha$ de $s_\alpha$ à isomorphisme près).

\begin{defi}
On dit qu'une représentation lisse admissible absolument irréductible $\pi$ de $L(F)$ sur $\ke$ est \emph{supersingulière} si la représentation lisse admissible irréductible $\Fpbar \otimes_{\ke} \pi$ de $L(F)$ sur $\Fpbar$ est supersingulière (voir \cite[Definition 4.7]{Her}).
\end{defi}

\begin{rema}
Les caractères lisses $F^\times \to \ke^\times$ sont des représentations supersingulières de $\GL_1(F)$.
\end{rema}

La conjecture suivante a été suggérée par Breuil lorsque $G=\GL_n$.
Les extensions sont calculées dans les catégories de représentations lisses admissibles sur $\ke$.
Dans les points (iii) et (iv), la condition \og Sinon \fg{} signifie que les conditions du point (ii) ne sont pas toutes satisfaites.

\begin{conj} \label{conj:ext}
Soient $P,P' \subset G$ deux sous-groupes paraboliques standards, $L \subset P,L' \subset P'$ les sous-groupes de Levi standards et $\pi,\pi'$ des représentations supersingulières de $L(F),L'(F)$ respectivement sur $\ke$.
On suppose $\IndF[P^-] \pi, \IndF[P'^-] \pi'$ irréductibles ou $p \neq 2$.
\begin{enumerate}
\item Si $P \not \subset P'$ et $P' \not \subset P$, alors
\begin{equation*}
\Ext^1_{G(F)} \left( \IndF[P'^-] \pi',\IndF[P^-] \pi \right) = 0.
\end{equation*}
\item Si $F=\Qp$, $P'=P$ et $\pi' \cong \pi^\alpha \otimes (\oma) \not \cong \pi$ avec $\alpha \in \Delta_L^\perp$, alors
\begin{equation*}
\dim_{\ke} \Ext^1_{G(\Qp)} \left( \IndQp[P^-] \pi',\IndQp[P^-] \pi \right) = 1.
\end{equation*}
\item Sinon si $P' \subset P$, alors le foncteur $\IndF[P^-]$ induit un isomorphisme $\ke$-linéaire
\begin{equation*}
\Ext^1_{L(F)} \left( \Ind_{(P'^- \cap L)(F)}^{L(F)} \pi',\pi \right) \iso \Ext^1_{G(F)} \left( \IndF[P'^-] \pi',\IndF[P^-] \pi \right).
\end{equation*}
\item Sinon si $P \subset P'$, alors le foncteur $\IndF[P'^-]$ induit un isomorphisme $\ke$-linéaire
\begin{equation*}
\Ext^1_{L'(F)} \left( \pi',\Ind_{(P^- \cap L')(F)}^{L'(F)} \pi \right) \iso \Ext^1_{G(F)} \left( \IndF[P'^-] \pi',\IndF[P^-] \pi \right).
\end{equation*}
\end{enumerate}
\end{conj}

\begin{rema}
Sous les conditions du point (ii), il n'existe pas d'extension non scindée de $\pi'$ par $\pi$ (car leurs caractères centraux sont distincts) mais on peut construire une extension non scindée entre leurs induites par induction parabolique à partir d'une extension non scindée entre deux séries principales de $\GL_2(\Qp)$.
\end{rema}

On démontre la conjecture \ref{conj:ext} pour les extensions par une série principale (sous des hypothèses de généricité lorsque $F=\Qp$).
Pour toute représentation lisse localement admissible $\pi$ de $L(F)$ sur $\ke$ et pour tout caractère lisse $\chi : T(F) \to \ke^\times$, on a une suite exacte de $\ke$-espaces vectoriels
\begin{multline} \label{SEadj}
0 \to \Ext^1_{L(F)} \left( \pi,\IndFL \chi \right) \to \Ext^1_{G(F)} \left( \IndF[P^-] \pi,\IndF \chi \right) \\
\to \Hom_{L(F)} \left( \pi,\HOrdF[1] \left( \IndF \chi \right) \right)
\end{multline}
(voir la suite exacte \cite[(3.7.6)]{Em2} avec $A=\ke$, $U=\pi$ et $V=\IndF \chi$ en tenant compte de l'isomorphisme \eqref{Ord} avec $A=\ke$ et $U=\chi$).

\begin{prop} \label{prop:conjQp}
On suppose $F=\Qp$.
Soient $\pi$ une représentation lisse admissible absolument irréductible de $L(\Qp)$ sur $\ke$ et $\chi : T(\Qp) \to \ke^\times$ un caractère lisse.
On suppose d'une part $L=T$ ou $\pi_{N_L(\Qp)}=0$ ; d'autre part $\IndQp \chi$ irréductible ou $p \neq 2$.
\begin{enumerate}
\item Si $P=B$ et $\pi = s_\alpha(\chi) \cdot (\oma) \neq \chi$ avec $\alpha \in \Delta$, alors
\begin{equation*}
\dim_{\ke} \Ext^1_{G(\Qp)} \left( \IndQp \pi,\IndQp \chi \right) = 1.
\end{equation*}
\item Sinon si $(s_\alpha(\chi) \cdot (\oma)) \circ \beta^\vee \neq 1$ pour tous $\alpha \in \Delta-\Delta_L$ et $\beta \in \Delta_L$, alors le foncteur $\IndQp[P^-]$ induit un isomorphisme $\ke$-linéaire
\begin{equation*}
\Ext^1_{L(\Qp)} \left( \pi,\IndQpL \chi \right) \iso \Ext^1_{G(\Qp)} \left( \IndQp[P^-] \pi,\IndQp \chi \right).
\end{equation*}
\end{enumerate}
\end{prop}

\begin{rema}
Si $\pi$ est supersingulière, alors $\pi$ est supercuspidale d'après \cite[Corollary 1.2 (ii)]{Her} lorsque $L=\GL_n$ et \cite[Theorem 1.2]{Abe} dans le cas général déployé.
On a donc $L=T$ ou $\pi_{N_L(\Qp)}=0$ par réciprocité de Frobenius (car $\pi_{N_L(\Qp)}$ est admissible de longueur finie d'après \cite[Corollary 3.6.7]{Em2} et \cite[Theorem 2.3.8 (1)]{Em1}).
\end{rema}

\begin{proof}
On suppose $P=B$.
Comme $\pi$ est absolument irréductible et $T(\Qp)$ est commutatif, $\pi$ est de dimension $1$.
Le point (i) est l'analogue modulo $p$ de \cite[Théorème 1.1 (ii)]{JH}.
Si $\pi \neq \chi$, alors la source et le but de l'isomorphisme du point (ii) sont nuls d'après \cite[Proposition 5.1.4 (i)]{JH} et l'analogue modulo $p$ de \cite[Théorème 1.1 (i)]{JH} respectivement.
Si $\pi = \chi$, alors le point (ii) se déduit de \cite[Théorème 1.2]{JHC} en remarquant que lorsque $p=2$, $\IndQp \chi$ est irréductible par hypothèse, donc $s_\alpha(\chi) \neq \chi$ pour tout $\alpha \in \Delta$ (voir \cite[Remarque 3.1.3]{JHC}).

On suppose $P \neq B$ et $(s_\alpha(\chi) \cdot (\oma)) \circ \beta^\vee \neq 1$ pour tous $\alpha \in \Delta-\Delta_L$ et $\beta \in \Delta_L$.
Les représentations $\IndQpL s_\alpha(\chi) \cdot (\oma)$ avec $\alpha \in \Delta-\Delta_L$ sont irréductibles d'après \cite[Théorème 4]{Oll} lorsque $L=\GL_n$ et \cite[Theorem 1.3]{Abe} dans le cas général déployé.
En utilisant la proposition \ref{prop:HnOrdirr}, on en déduit du théorème \ref{theo:HnOrd} avec $n=1$, $A=\ke$ et $U=\chi$ un isomorphisme $L(\Qp)$-équivariant
\begin{equation*}
\bigoplus_{\alpha \in \Delta-\Delta_L} \IndQpL s_\alpha(\chi) \cdot (\oma) \iso \HOrdQp \left( \IndQp \chi \right).
\end{equation*}
Or $\pi_{N_L(\Qp)}=0$ par hypothèse, donc par réciprocité de Frobenius on a
\begin{equation*}
\Hom_{L(\Qp)} \left( \pi,\IndQpL s_\alpha(\chi) \cdot (\oma) \right) = 0
\end{equation*}
pour tout $\alpha \in \Delta-\Delta_L$.
En utilisant la suite exacte \eqref{SEadj}, on en déduit le point (ii).
\end{proof}

\begin{prop} \label{prop:conjF}
On suppose $F \neq \Qp$.
Soient $\pi$ une représentation lisse localement admissible de $L(F)$ sur $\ke$ et $\chi : T(F) \to \ke^\times$ un caractère lisse.
Le foncteur $\IndF[P^-]$ induit un isomorphisme $\ke$-linéaire
\begin{equation*}
\Ext^1_{L(F)} \left( \pi,\IndFL \chi \right) \iso \Ext^1_{G(F)} \left( \IndF[P^-] \pi,\IndF \chi \right).
\end{equation*}
\end{prop}

\begin{proof}
L'isomorphisme se déduit de la suite exacte \eqref{SEadj} en utilisant le corollaire \ref{coro:HnOrdF} avec $n=1$,  $A=\ke$ et $U=\chi$.
\end{proof}

\begin{rema}
Par réduction modulo $\pe^k$ et dévissage, on prouve les résultats analogues $p$-adiques (c'est-à-dire dans les catégories de représentations continues unitaires admissibles sur $E$) avec $\pi$ résiduellement de longueur finie.
Pour l'analogue $p$-adique de la proposition \ref{prop:conjQp}, il faut supposer $s_\alpha(\chi) \cdot (\varepsilon^{-1} \circ \alpha) \neq \chi$ pour tout $\alpha \in \Delta$ ou $p \neq2$, et les représentations $\IndQpL s_\alpha(\chi) \cdot (\varepsilon^{-1} \circ \alpha)$ avec $\alpha \in \Delta-\Delta_L$ topologiquement irréductibles.
\end{rema}

\bibliographystyle{alpha-fr}
\bibliography{conjecture}

\end{document}